\documentclass[11pt]{article}

\usepackage{amsmath,amssymb}
\usepackage{amstext,mathrsfs,mathtools}
\usepackage{amsthm}
\usepackage{amsfonts}
\usepackage{bbm}

\usepackage{ascmac}

\usepackage[dvipdfmx]{graphicx}
\usepackage{authblk}
\usepackage{ulem}

\usepackage{cite}

\def\R{\mathbb{R}}
\def\C{\mathbb{C}}

\def\P{{\rm P}}
\def\E{{\rm E}}
\def\F{\mathcal{F}}

\def\1{\mathbbm{1}}

\def\re{{\rm Re}}
\def\im{{\rm Im}}
\def\tr{{\rm tr}}
\def\Ov{\mathcal{O}}
\def\lam{\pmb{\lambda}}

\theoremstyle{definition}

\newtheorem{theorem}{Theorem}[section]
\newtheorem{lemma}[theorem]{Lemma}
\newtheorem{prop}[theorem]{Proposition}
\newtheorem{cor}[theorem]{Corollary}

\newtheorem{rem}[theorem]{Remark}

\makeatletter
 
 \@addtoreset{equation}{section}
\makeatother

\title{Eigenvalue processes of Elliptic Ginibre Ensemble \\and their Overlaps}
\author{Satoshi Yabuoku\footnote{Department of Mathematics and Informatics, Graduate School of Science and Engineering, Chiba University, 1-33 Yayoi-cho, Inage-ku, Chiba 263-8522, JAPAN. e-mail: aasa1956@chiba-u.jp}}
\date{\today}

\begin{document}
\allowdisplaybreaks

\maketitle
\pagenumbering{roman}


\begin{abstract}
We consider the non-hermitian matrix-valued process of Elliptic Ginibre ensemble. This model includes Dyson's Brownian motion model and the time evolution model of Ginibre ensemble by using hermiticity parameter $\tau$. We show the complex eigenvalue processes satisfy the stochastic differential equations which are very similar to Dyson's model and give an explicit form of overlap correlations. As a corollary, in the case of $2 \times 2$ matrix, we also mention the relation between the diagonal overlap, which is the speed of eigenvalues, and the distance of the two eigenvalues.
\end{abstract}


\pagenumbering{arabic}
\section{Introduction}\

In random matrix theory, the study of eigenvalue processes was started by Dyson \cite{Dyson}. He considered the hermitian matrix-valued process whose entries are given by independent Brownian motions and derived the stochastic differential equations of the eigenvalues by perturbation theory:
\begin{align}\label{11}
d\lambda_i(t)=dB_i(t)+\frac{\beta}{2}\sum_{j\neq i}\frac{1}{\lambda_i(t)-\lambda_j(t)}dt\ ,\ \ i=1,\cdots,N
\end{align}
where $\beta=1,2,4$. The parameter $\beta$ implies the matrix symmetry and corresponds to Gaussian orthogonal, unitary and symplectic ensembles (GOE, GUE and GSE) as $\beta=1,2,4$, respectively \cite{Mehta, AGZ}. These processes are called Dyson's Brownian motion model, and when $\beta=2$, we call simply them Dyson's model in this paper. For $\beta \ge 1$, the processes satisfying \eqref{11} are non-colliding \cite{RS}, and so are the above three eigenvalue processes. This fact is naturally expected because the eigenvalues exert ``repulsive" force on each other \cite{Tao}. For other studies of time dependent random matrices, Dyson also derived the eigenvalue processes of unitary matrices in the same paper \cite{Dyson}, and Bru derived that of the positive definite hermitian matrix which is called Wishart processes \cite{Bru,Bru2}. Moreover, the relation between eigenvalue processes and non-colliding diffusion particle systems was reported in \cite{KT1,KT2,KT3, GM}. We remark that all the above matrices are normal and their eigenvalue processes are diffusion.

The aim of this paper is to derive the eigenvalue processes of Elliptic Ginibre ensemble (EGE) and show the relation between Dyson's model and the time evolution model of Ginibre ensemble with overlaps. EGE is one of the non-hermitian random matrix models, and recently, many applications of non-hermitian random matrices are discussed in physics: for example,  resonance scattering of quantum waves in open chaotic systems, quantum chromodynamics at non-zero chemical potential \cite{Oxford} and neural network dynamics \cite{CS}. We note that the matrix model used in the third example is very similar to EGE. To return to our subject, we begin to explain the statistic result for EGE and overlaps in random matrix theory. EGE was introduced as an interpolation between hermitian and non-hermitian matrices by  $J:=\sqrt{1+\tau}H_1+\sqrt{-1}\sqrt{1-\tau}H_2,\ -1\le\tau\le1$ in \cite{SCSS}. Here, $H_1$ and $H_2$ are independent GUE, that is, distributed in the hermitian matrix space in $\R^{N^2}$ with density $p(H)\varpropto \exp[-N\tr(H^2)]$. The parameter $\tau$ implies the degree of hermiticity. With $\tau=1$, the matrix $J$ is hermitian and thus GUE, and with $\tau=0$, it is completely non-hermitian and thus Ginibre ensemble  \cite{Gin, Mehta}. For $-1 < \tau < 1$, $J$ is distributed in the complex matrix space $\C^{N^2}$ with density $p(J) \varpropto \exp\left[-\frac{N}{1-\tau^2}\tr(JJ^*-\frac{\tau}{2}(J^2+{J^*}^2)) \right]$. The joint probability density function (jpdf for short) of the eigenvalues of $J$ is described explicitly as $$p(z_1,\cdots,z_N) \varpropto \exp\Bigl[-\frac{N}{1-\tau^2}(\sum_{i=1}^{N}|z_i|^2-\frac{\tau}{2}\sum_{i=1}^{N}(z^2+{z^*}^2))\Bigr]\prod_{i<j}|z_i-z_j|^2$$ in \cite{LS, Oxford}. For fixed $-1<\tau<1$, the limiting empirical spectral distribution converges to the uniform distribution on the ellipse $\Bigl\{z \in \C\ ;\left(\frac{\re(z)}{1+\tau}\right)^2+\left(\frac{\im(z)}{1-\tau}\right)^2\le1\Bigr\}$ in \cite{SCSS}. This convergence is known as ``elliptic law" \cite{Girko}. The condition $-1<\tau<1$ is called ``strong non-hermiticity" since the anti-hermitian part $\sqrt{-1}\sqrt{1-\tau}H_2$ of $J$ is the same order in $N$ as the hermitian part. In contrast, the limiting behavior under the condition of ``weak non-hermiticity" is also known. Suppose that $1-\tau=\frac{\alpha}{N}$ by some $\alpha>0$. Then in the limit $N \to \infty$, the density of the eigenvalues $z=x+\sqrt{-1}y$ of $J$ behaves asymptotically as $p_{sc}(x)p(y)$, where $p_{sc}(x)=\frac{1}{2\pi}\sqrt{4-x^2}$ on the interval $[-2,2]$ (Wigner's semicircle distribution) and $p(y)=\frac{1}{\sqrt{2\pi\alpha}}\exp(-\frac{N^2y^2}{2\alpha^2})$ (Gaussian distribution) . This result is first observed in \cite{FKS} by perturbation theory, but nowadays there are more detailed studies by using correlation function and kernels \cite{Oxford, AP}.

For non-normal matrices, the overlaps also have been studied. They are also called eigenvector correlations or condition numbers. For the right eigenvectors $R_j$ and left eigenvectors $L_j, j=1,\cdots,N$, the overlaps are defined by $\Ov_{ij}:=(R_j^*R_i)(L_j^*L_i)$. For normal matrices, overlaps are trivial, that is, $\Ov_{ij}\equiv \delta_{ij}$; on the contrary, for non-normal cases, they play an important role because the non-orthogonality of  eigenvectors effects the behavior of eigenvalues \cite{TE}. In the case of Ginibre ensemble, an early observation was given by Chalker and Mehlig. They estimated the asymptotic behavior of the conditional expectation $\E[\Ov_{11}|\lambda_1=z]\sim N(1-|z|^2)$ as $N \to \infty$ in \cite{CM}. Recently,  Bourgade and Dubach showed the limiting conditional distribution of $\frac{\Ov_{11}}{N}$ converges to inverse gamma distribution for complex Ginibre ensemble in \cite{BD} with probabilistic approach, and Fyodorov also showed a similar result for real Ginibre ensemble with supersymmetric approach in \cite{Fyo}. Furthermore, the result for real EGE was also reported in \cite{FW}.

The study of matrix-valued process for non-normal matrices is lesser than that for normal case since the eigenvectors and overlaps should be also concerned together as mentioned above. Nevertheless, there are remarkable results for Ginibre ensemble. This model is non-symmetric matrix-valued process whose entries are given by independent complex Brownian motions. Grela and Warcho\l \ solved the Fokker-Plank equation of the jpdf of the eigenvalues and eigenvectors \cite{GW,Dysonian} and observed the correlation between the distance of two eigenvalues and the diagonal overlap by numerical experiment with cooperators \cite{OU}. Bourgade and Dubach mentioned in \cite{BD} that the complex quadratic variations of the eigenvalue processes are truly the overlaps, and they also suggested the limiting behavior of the eigenvalues as the matrix size goes to infinity. In the above two results, they also pointed out that the stochastic differential equations of the eigenvalue processes have no drift term, without using their explicit forms; the eigenvalue processes are complex martingales. This is an unexpected and counterintuitive fact because eigenvalues should exert ``repulsive" force on each other by random matrix theory, and this force appear in the drift terms  for normal matrix cases as in \eqref{11}. For this reason, overlaps are very important quantities to understand the behavior of eigenvalue processes well. 

The main result of this paper is to give the stochastic differential equations \eqref{42}, with parameter $-1 \le \tau \le 1$, that  the eigenvalue processes of EGE satisfy. On the basis of the above results and observations, we consider the non-normal matrix valued process of EGE whose entries are given by independent Brownian motions. This model naturally gives an interpolation between Dyson's model and the time evolution of Ginibre ensemble by using hermiticity parameter $\tau$ in the same way as the statistic case. In the main theorem (Theorem \ref{thm1}), we show that for $ -1 \le \tau \le 1$, the eigenvalue processes of EGE satisfy the stochastic differential equations which have the drift of Dyson's model except for  $\tau=0$ and also show the explicit form of their time-depending overlaps described by given Brownian motions. As a result, we obtain the complex martingales of the eigenvalue processes of Ginibre ensemble explicitly, which tell us the interaction of the eigenvalues by the form of difference product, that is, Vandermonde determinant. In the case of $2 \times 2$ matrix, we can show that the quadratic variation of the diagonal overlap and the distance of the two eigenvalues is negative by using our explicit forms, which proves the fact in \cite{OU}; as the two eigenvalues get closer to each other, they move faster. We also show for $-1\le \tau \le 1$, the eigenvalue processes of EGE are non-colliding.

The organization of the paper is the following. In section 2, we show our main theorem and corollaries with some observations. In section 3, we give the proof of Theorem \ref{thm1} and Corollary \ref{cor2}. We put together some properties of characteristic polynomials, eigenvalues and determinants in Appendix.

\section{Settings and Main Results}

We consider the $N \times N$ matrix-valued process for EGE and define this model as follows:
\begin{align}\label{41}
J(t):=\frac{\sqrt{1+\tau}}{\sqrt{2}}H_1(t)+\sqrt{-1}\frac{\sqrt{1-\tau}}{\sqrt{2}}H_2(t), \ \ t \ge 0 , \ \ -1 \le \tau \le 1, 
\end{align}
where
\begin{align*}
(H_1(t))_{k\ell}:=\begin{dcases}
B_{kk}(t) & k=\ell\\
\frac{B_{k\ell}^R(t)+\sqrt{-1}B_{k\ell}^I(t)}{\sqrt{2}} & k<\ell\\
\overline{(H_1(t))_{\ell k}} & k>\ell\\
\end{dcases},\ \ 
(H_2(t))_{k\ell}:=\begin{dcases}
b_{kk}(t) & k=\ell\\
\frac{b_{k\ell}^R(t)+\sqrt{-1}b_{k\ell}^I(t)}{\sqrt{2}} & k<\ell\\
\overline{(H_2(t))_{\ell k}} & k>\ell\\
\end{dcases}.
\end{align*}
Here, $B_{kk},B_{k\ell}^R,B_{k\ell}^I,b_{kk},b_{k\ell}^R,b_{k\ell}^I\ (1 \le k\le \ell \le N)$ are independent one-dimensional Brownian motions defined on a filtered probability space $(\Omega,\F,\{\F_t\}_{t\ge0},\P)$. The entries of $J(t)$ have the correlation described by the complex quadratic variations:
\begin{align}\label{corr}
d\langle J_{ij},\overline{J_{k\ell}}\rangle_t=\delta_{ik}\delta_{j\ell}dt,\ d\langle J_{ij},J_{k\ell}\rangle_t=\tau\delta_{i\ell}\delta_{jk}dt.
\end{align} 
Here, for complex semi-martingales $M(t)=M^R(t)+\sqrt{-1}M^I(t)$ and $N(t)=N^R(t)+\sqrt{-1}N^I(t)$, the complex quadratic variation is defined as 
\begin{align}\label{cqv}
d\langle M,N\rangle_t:=d\langle M^R,N^R\rangle_t-d\langle M^I,N^I\rangle_t+\sqrt{-1}\bigl(d\langle M^R,N^I\rangle_t+d\langle M^I,N^R\rangle_t\bigr).
\end{align}
The quantities \eqref{corr} express the hermiticity of the matrix $J(t)$ by $\tau$. By construction of $J(t)$, we get Dyson's model with $\tau=1$ and Ginibre dynamics with $\tau=0$. $J(t)$ has pure imaginary eigenvalue processes with $\tau=-1$, which are just Dyson's model on the imaginary axis. Thus it is essential to consider $0 \le \tau \le 1$. From the perspective of normality of matrix, the case of $\tau=1$ and $\tau=0$ are extreme; in the former case each of the eigenvalue processes has the drift term, similar to that in \eqref{11}, which takes a larger absolute values as it gets closer to the other eigenvalues, and in the latter case the eigenvalue processes are complex martingale. We denote the eigenvalue processes of $J(t)$ by $\lam(t)=(\lambda_1(t), \cdots, \lambda_N(t))$. As mentioned above, these processes usually take complex values, so we write $\lambda_i(t)=\lambda_i^R(t)+\sqrt{-1}\lambda_i^I(t),\ i=1,\cdots,N$. We assume the following initial condition:
 \begin{align}\label{21}
  J(0)\ {\rm has\ simple\ spectrum},\ {\rm that\ is},\ \lambda_1(0), \cdots, \lambda_N(0)\ {\rm are\ distinct}. 
\end{align}
We denote the $N \times N$ identity matrix by $I$, and for a square matrix $A$, we define the $(N-1) \times (N-1)$ minor matrix $A_{k|\ell}$ that is obtained by removing the $k$-th row and the $\ell$-th column from A. Then we have the following result.
\begin{theorem}\label{thm1} 
For $-1 \le \tau \le 1$ with the initial condition \eqref{21}, the eigenvalue process $\lam(t)=(\lambda_1(t), \cdots, \lambda_N(t))$ of $J(t)$ is complex semi-martingale and satisfies the stochastic differential equations: 
\begin{align}\label{42}
d\lambda_i(t)
=\sum_{k,\ell=1}^N\frac{(-1)^{k+\ell}\det((\lambda_i(t)I-J(t))_{k|\ell})}{\prod_{j(\neq i)}(\lambda_i(t)-\lambda_j(t))}dJ_{k\ell}(t)+\tau\sum_{j(\neq i)}\frac{1}{\lambda_i(t)-\lambda_j(t)}dt,\ t \ge 0,\ i=1,\cdots,N
\end{align}
and the real  quadratic variations are given by
\begin{eqnarray}
& &\ \ \begin{aligned}\label{44}
d\langle \lambda_i^R\rangle_t=\frac{\Ov_{ii}(t)+\tau}{2}dt,\ d\langle \lambda_i^I\rangle_t=\frac{\Ov_{ii}(t)-\tau}{2}dt, \ \ d\langle \lambda_i^R,\lambda_i^I\rangle_t=0,
\end{aligned}\\
& &\begin{aligned}\label{45}
d\langle \lambda_i^R,\lambda_j^R\rangle_t&=d\langle \lambda_i^I,\lambda_j^I\rangle_t=\frac{\re(\Ov_{ij}(t))}{2}dt,\\
-d\langle \lambda_i^R,\lambda_j^I\rangle_t&=d\langle \lambda_i^I,\lambda_j^R\rangle_t=\frac{\im(\Ov_{ij}(t))}{2}dt,\ i\neq j
\end{aligned}
\end{eqnarray}
with
\begin{align}\label{46}
\Ov_{ij}(t):=\frac{\sum_{k=1}^N\det\Bigl(\bigl((\lambda_i(t)I-J(t))(\lambda_j(t)I-J(t))^*\bigr)_{k|k}\Bigr)}{\prod_{p(\neq i)}(\lambda_i(t)-\lambda_p(t))\prod_{q(\neq j)}\overline{(\lambda_j(t)-\lambda_q(t))}}.
\end{align}
Moreover, the eigenvalues do not collide each other, that is, 
\begin{align}\label{col}
T_{{\rm col}}:={\rm inf}\{ t>0;\ \lambda_i(t)=\lambda_j(t)\ {\rm for\ some}\ i \neq j\}=\infty, \ \text{a.s.}
\end{align}
\end{theorem}
By \eqref{44} and \eqref{45}, the complex quadratic variations of $\lam(t)$ are described as $d\langle \lambda_i,\overline{\lambda_j}\rangle_t=\Ov_{ij}(t)dt$. Although we do not use non-orthogonality of the eigenvectors  in our proof,  the quadratic variations coincide with the overlaps of $J(t)$ because the result and proof for the eigenvalue processes of Ginibre ensemble in \cite{BD} are also valid for our model. For this reason, we use the notation $\Ov_{ij}(t)$ for the quadratic variations.

\begin{rem}\label{rem1}
By \eqref{46} and Cauchy-Schwarz\ inequality, $\Ov_{ii}(t) \ge 1$ for $t \ge 0$. 
When $\tau=1$, \eqref{42} is truly Dyson's model \eqref{11} with $\beta=2$. Indeed, the overlaps \eqref{46} are  sensitive for the normality of matrices as shown below.
\begin{prop}\label{prop1}
For $\tau=1$, $\Ov_{ij}(t)\equiv\delta_{ij}$.
\end{prop}
\begin{proof}
$J(t)$ is hermitian with $\tau=1$, and so each of the matrices 
$$(\lambda_i(t)I-J(t))(\lambda_j(t)I-J(t))^*=(\lambda_i(t)I-J(t))(\lambda_j(t)I-J(t))$$
has the real eigenvalues $(\lambda_i(t)-\lambda_k(t))(\lambda_j(t)-\lambda_k(t)), k=1,\cdots,N$. If $i=j$, the matrix has only one zero eigenvalue, and for the numerator of \eqref{46}, by Lemma \ref{lemA2} we obtain
\begin{align*}
\sum_{k=1}^N\det\Bigl(\bigl((\lambda_i(t)I-J(t))(\lambda_j(t)I-J(t))^*\bigr)_{k|k}\Bigr)=\prod_{k(\neq i)}(\lambda_i(t)-\lambda_k(t))^2.
\end{align*}
On the other hand, for $i\neq j$, the matrix has two zero eigenvalues. Hence this summation vanishes.
\end{proof}
By Proposition \ref{prop1} and \eqref{44}-\eqref{46}, we have $\langle \lambda_i^R,\lambda_j^R\rangle_t=t\delta_{ij}$ and $\langle \lambda_i^I\rangle_t\equiv 0$ for $\tau=1$. Hence the martingale terms of $\lambda_1(t),\cdots,\lambda_N(t)$ are independent Brownian motions \cite{KS}, and we get Dyson's model. Similarly, we also obtain Dyson's model on the imaginary axis with $\tau=-1$.
\end{rem}
The parameter $\tau$ implies the hermiticity and controls the speeds of the real and imaginary parts of $\lambda_i(t)$. Theorem \ref{thm1} shows that for $\tau=0$, the drift term of $\lambda_i(t)$ completely vanishes, and this fact is observed in the previous study of Ginibre ensemble. Moreover, \eqref{44} states that each trajectory of the eigenvalue processes of Ginibre ensemble is Brownian motion, whereas they never collide. 
\begin{cor} \label{cor1}
Only for $\tau=0$, each of the eigenvalue processes is conformal martingale.  Hence for each $i=1,\cdots,N$, let
\begin{align*}
T_i(t):={\rm inf}\{ u\ge 0; \Ov_{ii}(u) >2t\},
\end{align*}
then $\lambda_i(T_i(t))$ is a complex Brownian motion on $(\Omega,\F,\{\F_{T_i(t)}\}_{t\ge0},\P)$.
\end{cor}

In the case of the matrix size $N=2$, the numerical experiment of the relation between the distance of the two eigenvalues $|\lambda_1(t)-\lambda_2(t)|$ and the diagonal overlap $\Ov_{11}(t)$ was observed in \cite{OU}. They reported that $\Ov_{11}(t)$, which is the speed of the eigenvalue processes, takes a larger value as the two eigenvalue processes get closer to each other. We attempt to justify this observation as follows. By \eqref{46}, we have the explicit forms of the overlaps in $N=2$: 
\begin{align}\label{435.5}
\Ov_{11}(t)=\frac{||J(t)||^2_2-\lambda_1(t)\overline{\lambda_2}(t)-\overline{\lambda_1}(t)\lambda_2(t)}{|\lambda_1(t)-\lambda_2(t)|^2},\ 
\Ov_{12}(t)=\frac{|\lambda_1(t)|^2+|\lambda_2(t)|^2-||J(t)||^2_2}{|\lambda_1(t)-\lambda_2(t)|^2}
\end{align}
where $||J(t)||^2_2:=\sum_{i,j=1}^2|J_{ij}(t)|^2$. The following relations hold:
\begin{align}\label{436} 
\Ov_{11}(t)=\Ov_{22}(t),\ \Ov_{12}(t)=\Ov_{21}(t),\ \Ov_{11}(t)+\Ov_{12}(t)=1.
\end{align}
We note that $\Ov_{11}(t)$ and $\Ov_{12}(t)$ are real valued process for $N=2$, nevertheless $\Ov_{ij}(t)$ are complex valued process for $N \ge 3$. By \eqref{436}, we need only to consider the diagonal overlap $\Ov_{11}(t)$.
\begin{cor}\label{cor2}
For $N=2$ and $-1 < \tau< 1$, $\Ov_{11}(t)$ satisfies the stochastic differential equation:
\begin{align*}
d\Ov_{11}(t)=&dM_{11}(t)+\frac{(2\Ov_{11}(t)-1)^2+1}{|\lambda_1(t)-\lambda_2(t)|^2}dt\\
&-\tau(2\Ov_{11}(t)-1)\biggl(\frac{1}{(\lambda_1(t)-\lambda_2(t))^2}+\frac{1}{(\overline{\lambda_1}(t)-\overline{\lambda_2}(t))^2}\biggr)dt,\ t \ge 0
\end{align*}
where $M_{11}$ is martingale and 
\begin{align*}
d\langle \Ov_{11}\rangle_t=&\frac{4\Ov_{11}(t)(2\Ov_{11}(t)-1)(\Ov_{11}(t)-1)}{|\lambda_1(t)-\lambda_2(t)|^2}dt\\
&-2\tau \Ov_{11}(t)(\Ov_{11}(t)-1)\biggl(\frac{1}{(\lambda_1(t)-\lambda_2(t))^2}+\frac{1}{(\overline{\lambda_1}(t)-\overline{\lambda_2}(t))^2}\biggr)dt.
\end{align*}
Moreover, the quadratic variation of $\Ov_{11}(t)$ and $|\lambda_1(t)-\lambda_2(t)|^2$ is 
\begin{align*}
\langle \Ov_{11},|\lambda_1-\lambda_2|^2\rangle_t=-8\int_0^t\Ov_{11}(s)(\Ov_{11}(s)-1)ds \le 0.
\end{align*}
\end{cor}
\begin{rem}\label{rem2}
For $\tau=1$ and $\tau=-1$, $J(t)$ is normal and $\Ov_{11}(t) \equiv 1$, so that the negative correlation between $\Ov_{11}(t)$ and $|\lambda_1(t)-\lambda_2(t)|^2$ vanishes. In particular, for the deterministic parameter $\tau$, $\Ov_{11}(t)$ takes the maximum value at $\tau=0$. To show this, we simply deal with the case of the initial condition that each of the eigenvalues starts at the origin, or equivalently, statistic EGE whose entries are given by independent centered gaussians with variance $t$. By schur decomposition, there exists a unitary matrix $U(t)$ such that 
$$U(t)^*J(t)U(t)=
\begin{pmatrix}
\lambda_1(t) & \sqrt{1-\tau^2}X\\
0& \lambda_2(t)
\end{pmatrix}$$
where $X$ is a complex Gaussian with mean 0 and variance $t$. Using this, we have 
$||J(t)||^2_2=\tr\bigl(J(t)J(t)^*\bigr)\overset{(d)}{=}|\lambda_1(t)|^2+|\lambda_2(t)|^2+(1-\tau^2)|X|^2$
and 
$$\Ov_{11}(t)=\frac{||J(t)||^2_2-2\re(\lambda_1(t)\overline{\lambda_2(t)})}{|\lambda_1(t)-\lambda_2(t)|^2}\overset{(d)}{=}1+\frac{t(1-\tau^2)Y}{|\lambda_1(t)-\lambda_2(t)|^2}$$
where $Y$ is independent of $\lambda_1(t), \lambda_2(t)$ and obeys exponential distribution with parameter 1. We notice that a similar deformation is obtained in \cite{BD} for Ginibre ensemble. Consequently, the complete non-normality at $\tau=0$ provides the biggest negative correlation of $\Ov_{11}(t)$ and $|\lambda_1(t)-\lambda_2(t)|^2$ and effects the behavior of the two eigenvalues significantly.
\end{rem}


\section{Proofs of main results}

\subsection{ Proof of Theorem 2.1}

 Firstly, we derive the stochastic differential equations \eqref{42} by implicit function theorem until the first collide time $t \in [0,T_{{\rm col}})$, and finally we show $T_{\rm col}=\infty,\ \text{a.s.}$  The detail calculations are summarized in subsection 3.2. We define the $N\times N$ deteriministic matrix $J$ as 
\begin{align*}
J=\frac{\sqrt{1+\tau}}{\sqrt{2}}H_1+\sqrt{-1}\frac{\sqrt{1-\tau}}{\sqrt{2}}H_2,\ \ -1 \le \tau \le 1
\end{align*}
where $H_1,H_2$ are hermitian: 
\begin{align*}
(H_1)_{k\ell}:=\begin{cases}
x_{kk} &  k=\ell\\
\frac{x_{k\ell}+\sqrt{-1}y_{k\ell}}{\sqrt{2}} & k<\ell\\
\overline{(H_1)_{\ell k}} & k>\ell\\
\end{cases},\ 
(H_2)_{k\ell}:=\begin{cases}
\alpha_{kk} & k=\ell \\
\frac{\alpha_{k\ell}+\sqrt{-1}\beta_{k\ell}}{\sqrt{2}} & k<\ell\\
\overline{(H_2)_{\ell k}} & k>\ell\\
\end{cases}.
\end{align*}
Here, $x_{k\ell},\alpha_{k\ell}\ (1\le k \le \ell \le N), y_{k\ell},\beta_{k\ell}\ (1\le k<\ell\le N)$ are real deterministic variables.
Let $f:\R^{2N^2+2}\to \C\cong \R^2$ be the characteristic polynomial of  $J$:
$$f(J,\lambda)=f(\lambda)=f(\lambda^R,\lambda^I):=\det(\lambda I-J)$$
where $\lambda:=\lambda^R+\sqrt{-1}\lambda^I \in \C$.  We denote $f=f^R+\sqrt{-1}f^I$ and the partial derivative of $f$ with respect to $\eta$ by $f_{\eta}$ and also define $f_R :=\frac{\partial f}{\partial \lambda^R}$ and $f_I :=\frac{\partial f}{\partial \lambda^I}$. Assume that $J$  has simple spectrum with the eigenvalues $\lambda_1,\cdots,\lambda_N$. $f$ is analytic with respect to $\lambda$, and so for all $i=1,\cdots,N$, the Jacovian is non-zero:
\begin{align*}
\det\left(\frac{\partial f}{\partial (\lambda^R,\lambda^I)}(\lambda_i)\right)=
\det\begin{pmatrix}
f^R_R(\lambda_i) & f^R_I(\lambda_i)\\
f^I_R(\lambda_i) & f^I_I(\lambda_i)
\end{pmatrix}=|f_{\lambda}(\lambda_i)|^2=\prod_{j (\neq i)}|\lambda_i-\lambda_j|^2 \neq 0.
\end{align*}
Hence we can apply implicit function theorem for each $\lambda_i$, and we obtain the derivative of $\lambda_i$ by using that of $f$ as follows: 
\begin{align}\label{31}
\frac{\partial \lambda_i^R}{\partial \eta}=-\frac{\re\bigl(f_\eta(\lambda_i)\overline{f_\lambda(\lambda_i)}\bigr)}{|f_{\lambda}(\lambda_i)|^2},\ \ 
\frac{\partial \lambda_i^I}{\partial \eta}=-\frac{\im\bigl(f_\eta(\lambda_i)\overline{f_\lambda(\lambda_i)}\bigr)}{|f_{\lambda}(\lambda_i)|^2}.
\end{align}
By using \eqref{31}, we can also calculate the second derivatives of $\lambda_i^R$ and $\lambda_i^I$ which give us the drift terms in the stochastic differential equations \eqref{42} and the quadratic variations in \eqref{44} and \eqref{45}. For a $C^2$ function $g: \R^{N^2} \to \R$,  we define the gradient and Laplacian of $g$ by
\begin{align*}
&\nabla g :=\left(\frac{\partial g}{\partial x_{11}},\frac{\partial g}{\partial x_{12}},\cdots,\frac{\partial g}{\partial x_{NN}},\frac{\partial g}{\partial \alpha_{11}},\frac{\partial g}{\partial \alpha_{12}},\cdots,\frac{\partial g}{\partial \alpha_{NN}},\frac{\partial g}{\partial y_{12}}\cdots,\frac{\partial g}{\partial y_{N-1 N}},\frac{\partial g}{\partial \beta_{12}}\cdots,\frac{\partial g}{\partial \beta_{N-1 N}}\right),\\
&\Delta g:=\sum_{k=1}^N\Bigl(\frac{\partial^2}{\partial x_{kk}^2}+\frac{\partial^2}{\partial \alpha_{kk}^2}\Bigr)g+\sum_{k<\ell}\Bigl(\frac{\partial^2}{\partial x_{k\ell}^2}+\frac{\partial^2}{\partial y_{k\ell}^2}+\frac{\partial^2}{\partial \alpha_{k\ell}^2}+\frac{\partial^2}{\partial \beta_{k\ell}^2}\Bigr)g.
\end{align*}
In the notation, a key lemma holds as follows.
\begin{lemma}\label{Lap} For $i=1,\cdots,N$, $\lambda_i^R$ and $\lambda_i^I$ are $C^2$ function. We have
\begin{align}\label{422}
\Delta \lambda_i^R=\tau\frac{\re(\overline{f_\lambda(\lambda_i)}f_{\lambda\lambda}(\lambda_i))}{|f_\lambda(\lambda_i)|^2},\ \ 
\Delta \lambda_i^I=\tau\frac{\im(\overline{f_\lambda(\lambda_i)}f_{\lambda\lambda}(\lambda_i))}{|f_\lambda(\lambda_i)|^2}.
\end{align}
Moreover, the inner products of the gradients of $\lambda_i^R$ and $\lambda_i^I$ are following:
\begin{eqnarray}
& &\begin{aligned}
\label{nabii}
&\nabla \lambda_i^R \cdot \nabla \lambda_i^R
=\frac{1}{2}\Biggl(\frac{\sum_{k,\ell=1}^N\Bigl|\det\bigl((\lambda_i I-J)_{k|\ell}\bigr)\Bigr|^2}{|f_\lambda(\lambda_i)|^2}+\tau\Biggr),\\
&\nabla \lambda_i^I \cdot \nabla \lambda_i^I
=\frac{1}{2}\Biggl(\frac{\sum_{k,\ell=1}^N\Bigl|\det\bigl((\lambda_i I-J)_{k|\ell}\bigr)\Bigr|^2}{|f_\lambda(\lambda_i)|^2}-\tau\Biggr),\\
&\nabla \lambda_i^R \cdot \nabla \lambda_i^I=0,
\end{aligned}\\
& &\begin{aligned}
\label{nabij}
\nabla \lambda_i^R \cdot \nabla \lambda_j^R&=\nabla \lambda_i^I \cdot \nabla \lambda_j^I=\frac{1}{2}\re\left(\frac{\sum_{k=1}^N\det\Bigl(((\lambda_i I-J)(\lambda_j I-J)^*)_{k|k}\Bigr)}{f_\lambda(\lambda_i)\overline{f_\lambda(\lambda_j)}}\right),\\
-\nabla \lambda_i^R \cdot \nabla \lambda_j^I&=\nabla \lambda_i^I \cdot \nabla \lambda_j^R=\frac{1}{2}\im\left(\frac{\sum_{k=1}^N\det\Bigl(((\lambda_i I-J)(\lambda_j I-J)^*)_{k|k}\Bigr)}{f_\lambda(\lambda_i)\overline{f_\lambda(\lambda_j)}}\right),\ \ i \neq j.
\end{aligned}
\end{eqnarray}
\end{lemma}
We give the proof of Lemma \ref{Lap} in subsection 3.2. As a result, we can derive the stochastic differential equations  \eqref{42} and the quadratic variations \eqref{44}-\eqref{46}. Under the assumption that $J$ has simple spectrum, we are able to use the above calculus and apply Ito's formula for $\lambda_i^R$ and $\lambda_i^I$  until first collision time $T_{\rm col}$ defined as \eqref{col}. Up to the time $T_{\rm col}$, we have
\begin{align}\label{424}
d\lambda_i(t)&=d\lambda_i^R(t)+\sqrt{-1}d\lambda_i^I(t)\nonumber\\
&=\sum_{k=1}^N\left(\frac{\partial \lambda_i^R}{\partial x_{kk}}+\sqrt{-1}\frac{\partial \lambda_i^I}{\partial x_{kk}}\right)dB_{kk}(t) +\sum_{k=1}^N\left(\frac{\partial \lambda_i^R}{\partial \alpha_{kk}}+\sqrt{-1}\frac{\partial \lambda_i^I}{\partial \alpha_{kk}}\right)db_{kk}(t)\nonumber\\
&+\sum_{k<\ell}\left\{\left(\frac{\partial \lambda_i^R}{\partial x_{k\ell}}+\sqrt{-1}\frac{\partial \lambda_i^I}{\partial x_{k\ell}}\right)dB_{k\ell}^R(t)+\left(\frac{\partial \lambda_i^R}{\partial y_{k\ell}}+\sqrt{-1}\frac{\partial \lambda_i^I}{\partial y_{k\ell}}\right)dB_{k\ell}^I(t)\right\}\nonumber\\
&+\sum_{k<\ell}\left\{\left(\frac{\partial \lambda_i^R}{\partial \alpha_{k\ell}}+\sqrt{-1}\frac{\partial \lambda_i^I}{\partial \alpha_{k\ell}}\right)db_{k\ell}^R(t)+\left(\frac{\partial \lambda_i^R}{\partial \beta_{k\ell}}+\sqrt{-1}\frac{\partial \lambda_i^I}{\partial \beta_{k\ell}}\right)db_{k\ell}^I(t)\right\}\nonumber\\
&\ +\frac{1}{2}\left(\Delta \lambda_i^R(t)+\sqrt{-1}\Delta \lambda_i^I(t)\right)dt.
\end{align}
For the local martingale part of \eqref{424}, we have $\frac{\partial \lambda_i^R}{\partial \eta}+\sqrt{-1}\frac{\partial \lambda_i^I}{\partial \eta}=-\frac{f_\eta(\lambda_i)}{f_\lambda(\lambda_i)}$ by \eqref{31}. Using \eqref{derixkk}, \eqref{derixkl} and \eqref{derideri} in the next subsection, we find
\begin{align*}
&\left(\frac{\partial \lambda_i^R}{\partial x_{kk}}+\sqrt{-1}\frac{\partial \lambda_i^I}{\partial x_{kk}}\right)dB_{kk}(t)+\left(\frac{\partial \lambda_i^R}{\partial \alpha_{kk}}+\sqrt{-1}\frac{\partial \lambda_i^I}{\partial \alpha_{kk}}\right)db_{kk}(t)
=\frac{\det\bigl((\lambda_i(t) I-J(t))_{k|k}\bigr)}{f_\lambda(\lambda_i(t))}dJ_{kk}(t),\\
&\left(\frac{\partial \lambda_i^R}{\partial x_{k\ell}}+\sqrt{-1}\frac{\partial \lambda_i^I}{\partial x_{k\ell}}\right)dB_{k\ell}^R(t)+\left(\frac{\partial \lambda_i^R}{\partial y_{k\ell}}+\sqrt{-1}\frac{\partial \lambda_i^I}{\partial y_{k\ell}}\right)dB_{k\ell}^I(t)\\
&+\left(\frac{\partial \lambda_i^R}{\partial \alpha_{k\ell}}+\sqrt{-1}\frac{\partial \lambda_i^I}{\partial \alpha_{k\ell}}\right)db_{k\ell}^R(t)+\left(\frac{\partial \lambda_i^R}{\partial \beta_{k\ell}}+\sqrt{-1}\frac{\partial \lambda_i^I}{\partial \beta_{k\ell}}\right)db_{k\ell}^I(t)\\
&=\frac{1}{f_\lambda(\lambda_i(t))}\Bigl((-1)^{k+\ell}\det\bigl((\lambda_i(t) I-J(t))_{k|\ell}\bigr)J_{k\ell}(t)+(-1)^{k+\ell}\det\bigl((\lambda_i(t) I-J(t))_{\ell|k}\bigr)J_{\ell k}(t)\Bigr).
\end{align*}
Therefore, the local martingale part of \eqref{424} is 
\begin{align}\label{425}
\sum_{k,\ell=1}^N\frac{(-1)^{k+\ell}\det\bigl((\lambda_i(t) I-J(t))_{k|\ell}\bigr)}{\prod_{j (\neq i)}(\lambda_i(t)-\lambda_j(t))}dJ_{k\ell}(t).
\end{align}
For the drift part of \eqref{424}, by \eqref{422} and Lemma \ref{lemA1}, we have
\begin{align}\label{426}
\frac{1}{2}\left(\Delta \lambda_i^R(t)+\sqrt{-1}\Delta \lambda_i^I(t)\right)
=\tau\frac{1}{2}\frac{f_{\lambda\lambda}(\lambda_i(t))}{f_\lambda(\lambda_i(t))}=\tau\sum_{j(\neq i)}\frac{1}{\lambda_i(t)-\lambda_j(t)}.
\end{align}
From \eqref{424}-\eqref{426}, we obtain the stochastic differential equations \eqref{42} for $t \in [0, T_{\rm col})$. Next, we derive the quadratic variations \eqref{44}-\eqref{46}. Because the local martingale part of $\lambda_i^R$ and $\lambda_i^I$ are constructed by $2N^2$ independent Brownian motions as \eqref{424}, we immediately find that
\begin{eqnarray*}
& &d\langle \lambda_i^R\rangle_t=\nabla \lambda_i^R(t) \cdot \nabla \lambda_i^R(t)dt,\ \ 
d\langle \lambda_i^I\rangle_t=\nabla \lambda_i^I(t) \cdot \nabla \lambda_i^I(t)dt,\\
& &d\langle \lambda_i^{\natural},\lambda_j^{\sharp}\rangle_t=\nabla \lambda_i^{\natural}(t) \cdot \nabla \lambda_j^{\sharp}(t)dt, \ \ \natural,\ \sharp = R,\ I.
\end{eqnarray*}
By Lemma \ref{lemA4}, we rewrite the summation of the numerator of \eqref{nabii} as
\begin{align*}
\sum_{k,\ell=1}^N\Bigl|\det\bigl((\lambda_i I-J)_{k|\ell}\bigr)\Bigr|^2&=\sum_{k,\ell=1}^N\det\bigl((\lambda_i I-J)_{k|\ell}\bigr)\det\bigl((\lambda_i I-J)^*_{\ell | k}\bigr)\\
&=\sum_{k=1}^N\det\Bigl(((\lambda_i I-J)(\lambda_i I-J)^*)_{k|k}\Bigr).
\end{align*}
Therefore, by \eqref{nabii} and \eqref{nabij}, \eqref{44}-\eqref{46} hold.
At the end of this subsection, we show that the stochastic differential equations \eqref{42} actually hold for $t \in [0, \infty)$.
\begin{prop}\label{prop3}
Under the condition \eqref{21}, $T_{\rm col}=\infty,\ \text{a.s.}$
\end{prop}
\begin{proof}
Rogers and Shi showed that collision time is infinity almost surely for general stochastic differential equations which include Dyson's model \cite{RS}. Hence we have only  to show the claim for $-1<\tau<1$. However, their method does not work for our complex eigenvalue processes straightforward because the martingale terms have the correlations \eqref{44}-\eqref{46}. Accordingly, we refer to the method in \cite{McKean, Bru}. Assume that $T_{\rm col} < \infty$ and define $U : \C^N \to \C$ as follows: 
\begin{align}\label{434}
U(z_1,\cdots, z_n):=\prod_{i<j}\frac{1}{z_i-z_j}.
\end{align}
$U$ is an analytic function with respect to $z_i,\ i=1,\cdots,N$  and the derivatives of $U$ are 
\begin{align*}
&\partial_{z_i}U=-\sum_{j(\neq i)}\frac{1}{z_i-z_j}U,\\
&\partial_{z_i}\partial_{z_i}U=2\sum_{j(\neq i)}\frac{1}{(z_i-z_j)^2}U+2\sum_{\substack{j<k\\ j,k\neq i}}\frac{1}{(z_i-z_j)(z_i-z_k)}U.
\end{align*}
From \eqref{44} and \eqref{45}, we obtain the complex\ quadratic\ variations of the eigenvalue processes as 
\begin{align}
\begin{aligned}\label{435}
&d\langle \lambda_i,\lambda_i \rangle_t=\tau dt,\\
&d\langle \lambda_i,\lambda_j \rangle_t=0,\ d\langle \lambda_i,\overline{\lambda_j} \rangle_t=\Ov_{ij}(t)dt, \ i \neq j.
\end{aligned}
\end{align}
We denote $d\lambda^{(i)}_t=dM^{(i)}_t+\tau\sum_{j(\neq i)}\frac{1}{\lambda^{(i)}_t-\lambda^{(j)}_t}dt,\ i=1,\cdots,N$. Applying Lemma \ref{lemA5} to $U$ and the eigenvalue processes for $t \in [0,T_{\rm col})$ with \eqref{435},  we obtain
\begin{align*}
dU(\lam_t)&=\sum_{i=1}^N\partial_{z_i}U(\lam_t)d\lambda^{(i)}_t+\frac{\tau}{2}\sum_{i=1}^N\partial_{z_i}\partial_{z_i}U(\lam_t)dt\\
&=-U(\lam_t)\sum_{i=1}^N\sum_{j(\neq i)}\frac{1}{\lambda^{(i)}_t-\lambda^{(j)}_t}dM^{(i)}_t
-\tau U(\lam_t)\sum_{i=1}^N\left(\sum_{j(\neq i)}\frac{1}{\lambda^{(i)}_t-\lambda^{(j)}_t}\right)^2dt\\
&\ \ \ +\frac{\tau}{2}U(\lam_t)\sum_{i=1}^N\left(2\sum_{j(\neq i)}\frac{1}{(\lambda^{(i)}_t-\lambda^{(j)}_t)^2}+
2\sum_{\substack{j<k\\ j,k\neq i}}\frac{1}{(\lambda^{(i)}_t-\lambda^{(j)}_t)(\lambda^{(i)}_t-\lambda^{(k)}_t)}\right)dt\\
&=-U(\lam_t)\sum_{i=1}^N\sum_{j(\neq i)}\frac{1}{\lambda^{(i)}_t-\lambda^{(j)}_t}dM^{(i)}_t
-\tau U(\lam_t)\sum_{i=1}^N\sum_{\substack{j<k\\ j,k\neq i}}\frac{1}{(\lambda^{(i)}_t-\lambda^{(j)}_t)(\lambda^{(i)}_t-\lambda^{(k)}_t)}dt.
\end{align*}
Indeed, the last summation vanishes.
\begin{lemma}\label{lem2}
\begin{align*}
\sum_{i=1}^N\sum_{\substack{j<k\\ j,k\neq i}}\frac{1}{(z_i-z_j)(z_i-z_k)}=0.
\end{align*}
\end{lemma}
\begin{proof}
\begin{align*}
(LHS)&=\sum_{i<j<k}\frac{1}{(z_i-z_j)(z_i-z_k)}+\sum_{j<i<k}\frac{1}{(z_i-z_j)(z_i-z_k)}+\sum_{j<k<i}\frac{1}{(z_i-z_j)(z_i-z_k)}\\
&=\sum_{i<j<k}\left(\frac{1}{(z_i-z_j)(z_i-z_k)}-\frac{1}{(z_i-z_j)(z_j-z_k)}+\frac{1}{(z_k-z_i)(z_k-z_j)}
\right)=0.
\end{align*}
\end{proof}
Therefore, we find $dU(\lam_t)=-U(\lam_t)\sum_{i=1}^N\sum_{j(\neq i)}\frac{1}{\lambda^{(i)}_t-\lambda^{(j)}_t}dM^{(i)}_t,\ t \in [0,T_{\rm col})$, and so $U(\lam_t)$ is a complex\ local\ martingale. We rewrite $U(\lam_t)=U^R_t+\sqrt{-1}U^I_t$. In the limit of $t \to T_{\rm col}$, by definition of $U$ the radial part diverges. Hence either one of the divergence holds:
\begin{itemize}
\item $\lim_{t \to T_{\rm col}}|U^R_t|=\infty$,
\item $\lim_{t \to T_{\rm col}}|U^I_t|=\infty$.
\end{itemize}
In the former case, $U^R_t$ is an one-dimensional local martingale whose real quadratic variation is
\begin{align*}
\langle U^R\rangle_t&=\frac{1}{2}\int_0^t\left(d\langle U,\overline{U}\rangle_s+\frac{d\langle U,U\rangle_s+d\langle \overline{U},\overline{U}\rangle_s}{2}\right)ds\\
&=\int_0^t\Biggl(|U(\lam_s)|^2\sum_i\ \Bigl|\sum_{k(\neq i)}\frac{1}{\lambda^{(i)}_s-\lambda^{(k)}_s}\Bigr|^2\Ov_{ii}(s)+2\tau\re\biggl(U(\lam_s)^2\sum_{k(\neq i)}\frac{1}{(\lambda^{(i)}_s-\lambda^{(k)}_s)^2}\biggr)\\
&\ \ \ \ \ \ +|U(\lam_s)|^2\sum_{i<j}\biggl(\sum_{k(\neq i)}\frac{1}{\lambda^{(i)}_s-\lambda^{(k)}_s}\biggr)\biggl(\sum_{\ell(\neq j)}\frac{1}{\overline{\lambda}^{(j)}_s-\overline{\lambda}^{(\ell)}_s}\biggr)\Ov_{ij}(s)
\Biggr)ds.
\end{align*}
We define $T^R(t):={\rm inf}\{u \ge0; \langle U^R\rangle_u >t\}$, and so the time changed process $B^R_t:=U^R_{T^R(t)},\ t \in [0,\langle U^R\rangle_{T_{\rm col}})$ is an one-dimensional Brownian motion in the usual manner. Hence 
\begin{align*}
\lim_{t \to \langle U^R\rangle_{T_{\rm col}}}|B^R_t|=\lim_{t \to T_{\rm col}}|U^R_t|=\infty
\end{align*}
which never occur by the properties of Brownian  motion's paths \cite{KS}. In the latter case, we also have the same contradiction by applying time change to $U^I_t$ with the real quadratic variation 
\begin{align*}
\langle U^I\rangle_t
&=\int_0^t\Biggl(|U(\lam_s)|^2\sum_i\ \Bigl|\sum_{k(\neq i)}\frac{1}{\lambda^{(i)}_s-\lambda^{(k)}_s}\Bigr|^2\Ov_{ii}(s)-2\tau\re\biggl(U(\lam_s)^2\sum_{k(\neq i)}\frac{1}{(\lambda^{(i)}_s-\lambda^{(k)}_s)^2}\biggr)\\
& \ \ \ \ \ \ +|U(\lam_s)|^2\sum_{i<j}\biggl(\sum_{k(\neq i)}\frac{1}{\lambda^{(i)}_s-\lambda^{(k)}_s}\biggr)\biggl(\sum_{\ell(\neq j)}\frac{1}{\overline{\lambda}^{(j)}_s-\overline{\lambda}^{(\ell)}_s}\biggr)\Ov_{ij}(s)\Biggr)ds.
\end{align*}
Therefore, we have the contradiction in both cases, and the claim holds.
\end{proof}
Together with Proposition \ref{prop3}, we complete the proof of Theorem \ref{thm1}.


\subsection{Proof of Lemma \ref{Lap}}
To show Lemma \ref{Lap}, we need to calculate the first and second derivatives of $\lambda_i^R$ and $\lambda_i^I$ which are described by those of $f$ as a result of implicit function theorem. For $\eta=x_{k\ell},\alpha_{k\ell},y_{k\ell},\beta_{k\ell}$, we apply chain rule to  the first derivatives of $\lambda_i^R$ and $\lambda_i^I$ in \eqref{31} and obtain
\begin{align}\label{32}
\frac{\partial^2 \lambda_i^R}{\partial \eta^2}=&
-\frac{\re\bigl(f_{\eta,\eta}(\lambda_i)\overline{f_\lambda(\lambda_i)}\bigr)
+2(\lambda_{i,\eta}^R)^2\re\bigl(\overline{f_\lambda(\lambda_i)}f_{\lambda\lambda}(\lambda_i)\bigr)-2\lambda_{i,\eta}^R \lambda_{i,\eta}^I\im\bigl(\overline{f_\lambda(\lambda_i)}f_{\lambda\lambda}(\lambda_i)\bigr)}{|f_{\lambda}(\lambda_i)|^2}\nonumber\\
&+\frac{2\re\bigl(\overline{(f_{\lambda}(\lambda_i)}^2f_{\lambda\eta}(\lambda_i)f_\eta(\lambda_i)\bigr)
 +\re\bigl(\overline{f_\lambda(\lambda_i)}f_{\lambda\lambda}(\lambda_i)\bigr)|f_\eta(\lambda_i)|^2}{|f_{\lambda}(\lambda_i)|^4},
 \end{align}
 \begin{align}\label{33}
\frac{\partial^2 \lambda_i^I}{\partial \eta^2}=&
 -\frac{\im\bigl(f_{\eta,\eta}(\lambda_i)\overline{f_\lambda(\lambda_i)}\bigr)
-2(\lambda_{i,\eta}^I)^2\im\bigl(\overline{f_\lambda(\lambda_i)}f_{\lambda\lambda}(\lambda_i)\bigr)
-2\lambda_{i,\eta}^R \lambda_{i,\eta}^I\re\bigl(\overline{f_\lambda(\lambda_i)}f_{\lambda\lambda}(\lambda_i)\bigr)}{|f_{\lambda}(\lambda_i)|^2}\nonumber\\
&+\dfrac{2\im\bigl(\overline{(f_{\lambda}(\lambda_i)}^2f_{\lambda\eta}(\lambda_i)f_\eta(\lambda_i)\bigr)
 -\im\bigl(\overline{f_\lambda(\lambda_i)}f_{\lambda\lambda}(\lambda_i)\bigr)|f_\eta(\lambda_i)|^2}{|f_{\lambda}(\lambda_i)|^4}.
\end{align}

Taking the summation in \eqref{32} and \eqref{33} for $\eta=x_{k\ell},\alpha_{k\ell},y_{k\ell},\beta_{k\ell}$, we yield
\begin{align}\label{213}
&\Delta \lambda_i^R=
-\dfrac{\re\bigl(\Delta f(\lambda_i)\overline{f_{\lambda}(\lambda_i)}\bigr)
+2\nabla \lambda_i^R \cdot \nabla \lambda_i^R\re\bigl(\overline{f_\lambda(\lambda_i)}f_{\lambda\lambda}(\lambda_i)\bigr)
-2 \nabla \lambda_i^R\cdot \nabla \lambda_i^I\im\bigl(\overline{f_\lambda(\lambda_i)}f_{\lambda\lambda}(\lambda_i)\bigr)}{|f_{\lambda}(\lambda_i)|^2}\nonumber\\
&+\dfrac{2\re\bigl(\overline{f_{\lambda}(\lambda_i)}^2\nabla f_{\lambda}(\lambda_i) \cdot \nabla f(\lambda_i)\bigr)
+\re\bigl(\overline{f_\lambda(\lambda_i)}f_{\lambda\lambda}(\lambda_i)\bigr)\left(\nabla f^R(\lambda_i) \cdot \nabla f^R(\lambda_i)+\nabla f^I(\lambda_i) \cdot \nabla f^I(\lambda_i)\right)}{|f_{\lambda}(\lambda_i)|^4},
\end{align}
\begin{align}\label{214} 
&\Delta \lambda_i^I=
 -\frac{\im\bigl(\Delta f(\lambda_i)\overline{f_{\lambda}(\lambda_i)}\bigr)
-2\nabla \lambda_i^I \cdot \nabla \lambda_i^I\im\bigl(\overline{f_\lambda(\lambda_i)}f_{\lambda\lambda}(\lambda_i)\bigr)
-2 \nabla \lambda_i^R\cdot \nabla \lambda_i^I\re\bigl(\overline{f_\lambda(\lambda_i)}f_{\lambda\lambda}(\lambda_i)\bigr)}{|f_{\lambda}(\lambda_i)|^2}\nonumber\\
&+\dfrac{2\im\bigl(\overline{f_{\lambda}(\lambda_i)}^2\nabla f_{\lambda}(\lambda_i) \cdot \nabla f(\lambda_i)\bigr)
-\im\bigl(\overline{f_\lambda(\lambda_i)}f_{\lambda\lambda}(\lambda_i)\bigr)\left(\nabla f^R(\lambda_i) \cdot \nabla f^R(\lambda_i)+\nabla f^I(\lambda_i) \cdot \nabla f^I(\lambda_i)\right)}{|f_{\lambda}(\lambda_i)|^4}.
\end{align}
From \eqref{31}, we also have the gradient terms of $\lambda^R_i$ and $\lambda^I_i$: 
\begin{align}\label{34}
&\nabla \lambda^R_i \cdot \nabla \lambda^R_i\nonumber\\
&=\frac{f^R_R(\lambda_i)^2\nabla f^R(\lambda_i) \cdot \nabla f^R(\lambda_i)+\im\bigl(f_{\lambda}(\lambda_i)^2\bigr)\nabla f^R(\lambda_i) \cdot \nabla f^I(\lambda_i)+f^I_R(\lambda_i)^2\nabla f^I(\lambda_i) \cdot \nabla f^I(\lambda_i)}{|f_{\lambda}(\lambda_i)|^4},\\
\label{35}
&\nabla \lambda^I_i \cdot \nabla \lambda^I_i\nonumber\\
&=\frac{f^I_R(\lambda_i)^2\nabla f^R(\lambda_i) \cdot \nabla f^R(\lambda_i)-\im\bigl(f_{\lambda}(\lambda_i)^2\bigr)\nabla f^R(\lambda_i) \cdot \nabla f^I(\lambda_i)+f^R_R(\lambda_i)^2f^I(\lambda_i) \cdot \nabla f^I(\lambda_i)}{|f_{\lambda}(\lambda_i)|^4},\\
\label{36}
&\nabla \lambda^R_i\cdot \nabla \lambda^I_i=\frac{\im\bigl(\overline{f_{\lambda}(\lambda_i)}^2\nabla f(\lambda_i) \cdot \nabla f(\lambda_i)\bigr)
}{2|f_{\lambda}(\lambda_i)|^4}.
\end{align}
To calculate the above quantities, we must know the derivatives of $f$ explicitly, and so we use Lemma \ref{lemA3} with $A=\lambda I-J$. We note that for $k<\ell$, each determinant in \eqref{A3} does not have $(k,k)$ and $(k,\ell)$ entries, and we obtain
\begin{equation}\label{derixkk}
f_{x_{kk}}=-\frac{\sqrt{1+\tau}}{\sqrt{2}}\det\bigl((\lambda I-J)_{k|k}\bigr),\ \ f_{x_{kk},x_{kk}}=0.
\end{equation}
For the off-diagonal entries of A, 
\begin{align*}
a_{k\ell}a_{\ell k}=\frac{(1+\tau)(x_{k\ell}^2+y_{k\ell}^2)-(1-\tau)(\alpha_{k\ell}^2+\beta_{k\ell}^2)+\sqrt{-1}\ 2\sqrt{1-\tau^2}(x_{k\ell}\alpha_{k\ell}+y_{k\ell}\beta_{k\ell})}{4}
\end{align*}
which gives
\begin{align}\label{derixkl}
f_{x_{k\ell}}&=-\frac{(1+\tau)2x_{k\ell}+\sqrt{-1}\ 2\sqrt{1-\tau^2}\alpha_{k\ell}}{4}\det(A_{k\ell | \ell k})
-\sum_{\substack{q\neq k,\ell\\q<\ell}}(-1)^{k+q-1}\frac{\sqrt{1+\tau}}{2}a_{\ell q}\det(A_{k\ell | \ell q})\nonumber\\
&-\frac{\sqrt{1+\tau}}{2}\sum_{\substack{q\neq k,\ell\\q>\ell}}(-1)^{k+q}a_{\ell q}\det(A_{k\ell | \ell q})
-\frac{\sqrt{1+\tau}}{2}\sum_{\substack{p\neq k,\ell\\p>k}}(-1)^{\ell+p-1}a_{kp}\det(A_{k\ell | pk})\nonumber\\
&-\frac{\sqrt{1+\tau}}{2}\sum_{\substack{p\neq k,\ell\\p<k}}(-1)^{\ell+p}a_{kp}\det(A_{k\ell | pk})\nonumber\\
&=-\frac{\sqrt{1+\tau}}{2}(-1)^{k+\ell}\Bigl(\det\bigl((\lambda I-J)_{k|\ell}\bigr)+\det\bigl((\lambda I-J)_{\ell |k}\bigr)\Bigr).
\end{align}
Expanding $\det\bigl((\lambda I-J)_{k|\ell}\bigr)$ and $\det\bigl((\lambda I-J)_{\ell |k}\bigr)$ by each the $\ell$-th and $k$-th row, we have
\begin{align*}
& \frac{\partial \det\bigl((\lambda I-J)_{k|\ell}\bigr)}{\partial x_{k\ell}}
 =\frac{\sqrt{1+\tau}}{2}(-1)^{k+\ell}\det\bigl((\lambda I-J)_{k\ell |k\ell}\bigr),\\
 & \frac{\partial \det\bigl((\lambda I-J)_{\ell |k}\bigr)}{\partial x_{k\ell}}
 =\frac{\sqrt{1+\tau}}{2}(-1)^{k+\ell}\det\bigl((\lambda I-J)_{k\ell | k\ell}\bigr),
\end{align*}
and we yield
\begin{align}\label{derixklxkl}
f_{x_{k\ell},x_{k\ell}}=-\frac{1+\tau}{2}\det\bigl((\lambda I-J)_{k\ell | k \ell}\bigr).
\end{align}
Similarly, we also have the other first and second derivatives of $f$:
\begin{align}
\begin{aligned}\label{derideri}
&f_{y_{k\ell}}=-\sqrt{-1}\frac{\sqrt{1+\tau}}{2}(-1)^{k+\ell}\Bigl(\det\bigl((\lambda I-J)_{k|\ell}\bigr)-\det\bigl((\lambda I-J)_{\ell |k}\bigr)\Bigr),\ \ 
f_{y_{k\ell},y_{k\ell}}=f_{x_{k\ell},x_{k\ell}},\\
&f_{\alpha_{kk}}=-\sqrt{-1}\frac{\sqrt{1-\tau}}{\sqrt{2}}\det\bigl((\lambda I-J)_{k|k}\bigr),\ f_{\alpha_{kk},\alpha_{kk}}=0,\\
&f_{\alpha_{k\ell}}=-\sqrt{-1}\frac{\sqrt{1-\tau}}{2}(-1)^{k+\ell}\Bigl(\det\bigl((\lambda I-J)_{k|\ell}\bigr)+\det\bigl((\lambda I-J)_{\ell |k}\bigr)\Bigr),\\
&f_{\alpha_{k\ell},\alpha_{k\ell}}=\frac{1-\tau}{2}\det\bigl((\lambda I-J)_{k\ell |k\ell}\bigr),\\
&f_{\beta_{k\ell}}=\frac{\sqrt{1-\tau}}{2}(-1)^{k+\ell}\Bigl(\det\bigl((\lambda I-J)_{k|\ell}\bigr)-\det\bigl((\lambda I-J)_{\ell |k}\bigr)\Bigr),\ \ 
f_{\beta_{k\ell},\beta_{k\ell}}=f_{\alpha_{k\ell},\alpha_{k\ell}}.
\end{aligned}
\end{align}
Using these derivatives and applying Lemma \ref{lemA4}, we get 
\begin{align}\label{39}
&\nabla f(\lambda) \cdot \nabla f(\lambda)=\tau\sum_{k,\ell =1}^N\det\bigl((\lambda I-J)_{k|\ell}\bigr)\det\bigl((\lambda I-J)_{\ell |k}\bigr)
=\tau\sum_{k=1}^N\det\bigl((\lambda I-J)^2_{k|k}\bigr),\\ \label{310}
&\nabla f(\lambda) \cdot \nabla \overline{f(\lambda)}=\sum_{k,\ell =1}^N\Bigl|\det\bigl((\lambda I-J)_{k|\ell}\bigr)\Bigr|^2.
\end{align}
Indeed, the summation of \eqref{39} has an useful expression.
\begin{lemma} \label{lem1}
For all $\lambda \in \C$,
\begin{align*}
\ \sum_{k=1}^N\det\bigl((\lambda I-J)^2_{k|k}\bigr)=f_\lambda(\lambda)^2-2\sum_{k<l}(\lambda-\lambda_k)(\lambda-\lambda_l)\prod_{m(\neq k,l)}(\lambda-\lambda_m)^2.
\end{align*}
In particular, if $\lambda_i$ is one of the eigenvalues of $J$, then
\begin{equation}\label{415}
\sum_{k=1}^N\det\bigl((\lambda_i I-J)^2_{k|k}\bigr)=f_\lambda(\lambda_i)^2.
\end{equation}
\end{lemma}
\begin{proof}
$(\lambda I-J)^2$ has the eigenvalues $(\lambda-\lambda_k)^2,\ k=1,\cdots,N$, and by Lemma \ref{lemA2}, we obtain
\begin{align}\label{416}
\sum_{k=1}^N\det\bigl((\lambda I-J)^2_{k|k}\bigr)=\sum_{k=1}^N\prod_{\ell(\neq k)}(\lambda-\lambda_\ell)^2.
\end{align}
On the other hand, since $f_\lambda(\lambda)=\sum_{k=1}^N\prod_{\ell(\neq k)}(\lambda-\lambda_\ell)$, the claim holds.
\end{proof}

Differentiating \eqref{39} with respect to $\lambda$ and using \eqref{416}, we have
\begin{align*}
&\nabla f_{\lambda}(\lambda) \cdot \nabla f(\lambda)=\tau\sum_{k \neq \ell}(\lambda-\lambda_\ell)\prod_{m(\neq k,\ell)}(\lambda-\lambda_m)^2.
\end{align*} 
Because $f_{\lambda \lambda}(\lambda_i)=2\sum_{k (\neq i)}\prod_{l (\neq i,k)} (\lambda_i-\lambda_l)$, we take $\lambda=\lambda_i$ and obtain
\begin{align}\label{452}
&\nabla f_{\lambda}(\lambda_i) \cdot \nabla f(\lambda_i)=\tau\sum_{k ( \neq i)}(\lambda_i-\lambda_k)\prod_{l(\neq i,k)}(\lambda_i-\lambda_l)^2
=\frac{\tau}{2}f_{\lambda}(\lambda_i)f_{\lambda \lambda}(\lambda_i).
\end{align}
Next, we calculate the Laplacian of $f$. Using the second derivatives in \eqref{derixkk}, \eqref{derixklxkl} and \eqref{derideri},
\begin{align*}
\Delta f(\lambda)&=-\sum_{k<\ell}(1+\tau)\det\bigl((\lambda I-J)_{k\ell | k\ell}\bigr)+\sum_{k<\ell}(1-\tau)\det\bigl((\lambda I-J)_{k\ell|k\ell}\bigr)\\
&=-2\tau\sum_{k<\ell}\det\bigl((\lambda I-J)_{k\ell |k\ell}\bigr).
\end{align*}
We apply Lemma \ref{lemA2} and obtain $\sum_{k<\ell}\det\bigl((\lambda I-J)_{k\ell|k\ell}\bigr)=\sum_{k \neq \ell}\prod_{m (\neq k,\ell)}(\lambda-\lambda_m)$. Taking $\lambda=\lambda_i$, we find a simple form:
\begin{align}\label{420}
\Delta f(\lambda_i)=-2\tau\sum_{k ( \neq i)}\prod_{l(\neq i,k)}(\lambda_i-\lambda_l)=-\tau f_{\lambda \lambda}(\lambda_i).
\end{align}

Using all the above equations, we prove Lemma \ref{Lap}.
\begin{proof}[Proof of Lemma \ref{Lap}]
Firstly, we show \eqref{422}. From \eqref{39} and \eqref{415}, the numerator of  \eqref{36} vanishes, and we get
\begin{align}\label{421}
\nabla \lambda^R_i \cdot \nabla \lambda^I_i=0.
\end{align}

For \eqref{213}, by \eqref{452}, \eqref{420} and \eqref{421}, we have 
\begin{align*}
\Delta \lambda_i^R=&
-\frac{-\tau\re\bigl(\overline{f_{\lambda}(\lambda_i)}f_{\lambda\lambda}(\lambda_i)\bigr)
+2\nabla \lambda_i^R \cdot \nabla \lambda_i^R\re\bigl(\overline{f_\lambda(\lambda_i)}f_{\lambda\lambda}(\lambda_i)\bigr)}{|f_{\lambda}(\lambda_i)|^2}\nonumber\\
&+\dfrac{\tau\re\bigl(\overline{f_{\lambda}(\lambda_i)}^2f_{\lambda}(\lambda_i)f_{\lambda \lambda}(\lambda_i)\bigr)
+\re\bigl(\overline{f_\lambda(\lambda_i)}f_{\lambda\lambda}(\lambda_i)\bigr)\left(\nabla f^R(\lambda_i) \cdot \nabla f^R(\lambda_i)+\nabla f^I(\lambda_i) \cdot \nabla f^I(\lambda_i)\right)}{|f_{\lambda}(\lambda_i)|^4}\\
&=\frac{\re\bigl(\overline{f_{\lambda}(\lambda_i)}f_{\lambda\lambda}(\lambda_i)\bigr)}{|f_{\lambda}(\lambda_i)|^2}
\left(2\tau-2\nabla \lambda_i^R \cdot \nabla \lambda_i^R+\frac{\nabla f^R(\lambda_i) \cdot \nabla f^R(\lambda_i)+\nabla f^I(\lambda_i) \cdot \nabla f^I(\lambda_i)}{|f_{\lambda}(\lambda_i)|^2}\right).
\end{align*}
To calculate the last term, we rewrite \eqref{39} 
by using $f^R(\lambda)$ and $f^I(\lambda)$ as
\begin{align*}
&\nabla f(\lambda) \cdot \nabla f(\lambda)=\nabla f^R(\lambda) \cdot \nabla f^R(\lambda)-\nabla f^I(\lambda) \cdot \nabla f^I(\lambda)+2\sqrt{-1}\nabla f^R(\lambda) \cdot \nabla f^I(\lambda).
\end{align*}
Applying \eqref{39} and \eqref{415}, we have
\begin{align}\label{453}
\begin{aligned}
&\nabla f^R(\lambda_i) \cdot \nabla f^R(\lambda_i)-\nabla f^I(\lambda_i) \cdot \nabla f^I(\lambda_i)=\tau\re( f_\lambda(\lambda_i)^2),\\
&\nabla f^R(\lambda_i) \cdot \nabla f^I(\lambda_i)=\frac{\tau}{2}\im( f_\lambda(\lambda_i)^2)
\end{aligned}
\end{align}
which gives
\begin{align}\label{454}
&-2\nabla \lambda_i^R \cdot \nabla \lambda_i^R
 +\dfrac{\nabla f^R(\lambda_i) \cdot \nabla f^R(\lambda_i)+\nabla f^I(\lambda_i) \cdot \nabla f^I(\lambda_i)}{|f_\lambda(\lambda_i)|^2}
=-\tau.
\end{align}
Here, we use \eqref{34}. Hence we obtain the former equation of \eqref{422}. Similarly, we calculate \eqref{214} and get
\begin{align*}
\Delta \lambda_i^I=&
 -\frac{-\tau\im\bigl(\overline{f_{\lambda}(\lambda_i)}f_{\lambda\lambda}(\lambda_i)\bigr)
-2\nabla \lambda_i^I \cdot \nabla \lambda_i^I\im\bigl(\overline{f_\lambda(\lambda_i)}f_{\lambda\lambda}(\lambda_i)\bigr)}{|f_{\lambda}(\lambda_i)|^2}\nonumber\\
&+\dfrac{\tau\im\bigl(\overline{f_{\lambda}(\lambda_i)}^2f_{\lambda}(\lambda_i)f_{\lambda \lambda}(\lambda_i)\bigr)
-\im\bigl(\overline{f_\lambda(\lambda_i)}f_{\lambda\lambda}(\lambda_i)\bigr)\left(\nabla f^R(\lambda_i) \cdot \nabla f^R(\lambda_i)+\nabla f^I(\lambda_i) \cdot \nabla f^I(\lambda_i)\right)}{|f_{\lambda}(\lambda_i)|^4}\\
=&\frac{\im\bigl(\overline{f_{\lambda}(\lambda_i)}f_{\lambda\lambda}(\lambda_i)\bigr)}{|f_{\lambda}(\lambda_i)|^2}
\left(2\tau+2\nabla \lambda_i^I \cdot \nabla \lambda_i^I-\frac{\nabla f^R(\lambda_i) \cdot \nabla f^R(\lambda_i)+\nabla f^I(\lambda_i) \cdot \nabla f^I(\lambda_i)}{|f_{\lambda}(\lambda_i)|^2}\right).
\end{align*}

From \eqref{35} and \eqref{453}, we also yield
\begin{align}\label{455}
2\nabla \lambda_i^I \cdot \nabla \lambda_i^I-\frac{\nabla f^R(\lambda_i) \cdot \nabla f^R(\lambda_i)+\nabla f^I(\lambda_i) \cdot \nabla f^I(\lambda_i)}{|f_{\lambda}(\lambda_i)|^2}=-\tau.
\end{align}
Hence we obtain the latter equation of \eqref{422}, and we finish the calculations of $\Delta \lambda_i^R$ and $\Delta \lambda_i^I$. Secondly, we show \eqref{nabii}. 
Note that 
\begin{align*}
&\nabla f(\lambda) \cdot \nabla \overline{f(\lambda)}=\nabla f^R(\lambda) \cdot \nabla f^R(\lambda)+\nabla f^I(\lambda) \cdot \nabla f^I(\lambda)
\end{align*}
and by \eqref{310}, \eqref{454}, \eqref{455} and \eqref{421}, we yield \eqref{nabii}. 
Finally, we show \eqref{nabij}. From \eqref{31}, for $i \neq j$, we have 
\begin{align*}
\nabla \lambda_i^R \cdot \nabla \lambda_j^R
=&\frac{f^R_R(\lambda_i)f^R_R(\lambda_j)\nabla f^R(\lambda_i) \cdot \nabla f^R(\lambda_j)
+f^R_R(\lambda_i)f^I_R(\lambda_j)\nabla f^R(\lambda_i) \cdot \nabla f^I(\lambda_j)}{|f_\lambda(\lambda_i)|^2|f_\lambda(\lambda_j)|^2}\nonumber\\
&+\frac{f^I_R(\lambda_i)f^R_R(\lambda_j)\nabla f^I(\lambda_i) \cdot \nabla f^R(\lambda_j)
+f^I_R(\lambda_i)f^I_R(\lambda_j)\nabla f^I(\lambda_i) \cdot \nabla f^I(\lambda_j)}{|f_\lambda(\lambda_i)|^2|f_\lambda(\lambda_j)|^2},
\end{align*}
and so we want the gradient terms. By straight computation, we get
\begin{align*}
&\nabla f^R(\lambda_i) \cdot \nabla f^R(\lambda_j)\\
&=\frac{1}{2}\re\left(\sum_{k=1}^N\det\Bigl(((\lambda_i I-J)(\lambda_j I-J)^*)_{k|k}\Bigr)\right)
+\frac{\tau}{2}\re\left(\sum_{k=1}^N\det\Bigl(((\lambda_i I-J)(\lambda_j I-J))_{k|k}\Bigr)\right).
\end{align*}
Here, we use $\overline{\det\bigl((\lambda_j I-J)_{i|j}\bigr)}=\det\bigl((\lambda_j I-J)^*_{j|i}\bigr)$ and Lemma \ref{lemA4} in the last equation. Since $(\lambda_i I-J)(\lambda_j I-J)$ must have two zero eigenvalues, the second summation vanishes by applying Lemma \ref{lemA2}. Therefore, we obtain
\begin{align}\label{430}
\nabla f^R(\lambda_i) \cdot \nabla f^R(\lambda_j)=\frac{1}{2}\re\left(\sum_{k=1}^N\det\Bigl(((\lambda_i I-J)(\lambda_j I-J)^*)_{k|k}\Bigr)\right).
\end{align}
Similarly, we also obtain
\begin{align}\label{431}
\nabla f^I(\lambda_i) \cdot \nabla f^I(\lambda_j)&=\nabla f^R(\lambda_i) \cdot \nabla f^R(\lambda_j),\\
\label{432}
-\nabla f^R(\lambda_i) \cdot \nabla f^I(\lambda_j)&=\nabla f^I(\lambda_i) \cdot \nabla f^R(\lambda_j)
=\frac{1}{2}\im\left(\sum_{k=1}^N\det\Bigl(((\lambda_i I-J)(\lambda_j I-J)^*)_{k|k}\Bigr)\right).
\end{align}
By \eqref{430}-\eqref{432}, we conclude
\begin{align*}
\nabla \lambda_i^R \cdot \nabla \lambda_j^R
&=\frac{1}{2}\re\left(\frac{\sum_{k=1}^N\det\Bigl(((\lambda_i I-J)(\lambda_j I-J)^*)_{k|k}\Bigr)}{f_\lambda(\lambda_i)\overline{f_\lambda(\lambda_j)}}\right).
\end{align*}
The others are also obtained by the same calculation:
\begin{align*}
\nabla \lambda_i^I \cdot \nabla \lambda_j^I&=\nabla \lambda_i^R \cdot \nabla \lambda_j^R,\\
-\nabla \lambda_i^R \cdot \nabla \lambda_j^I&=\nabla \lambda_i^I \cdot \nabla \lambda_j^R
=\frac{1}{2}\im\left(\frac{\sum_{k=1}^N\det\Bigl(((\lambda_i I-J)(\lambda_j I-J)^*)_{k|k}\Bigr)}{f_\lambda(\lambda_i)\overline{f_\lambda(\lambda_j)}}\right).
\end{align*}
Hence we also obtain \eqref{nabij}, and we finish the proof of Lemma \ref{Lap}.
\end{proof}


\subsection{Proof of Corollary \ref{cor2}}

From \eqref{435.5}, \eqref{436} and \eqref{435}, we apply Ito's formula to $\Ov_{11}$ and obtain 
\begin{align}
&d\bigl(||J||^2_2-\lambda_1\overline{\lambda_2}-\overline{\lambda_1}\lambda_2\bigr)(t)
=\sum_{i,j=1}^2(J_{ij}(t)d\overline{J_{ij}}(t)+\overline{J_{ij}}(t)dJ_{ij}(t))+2(\Ov_{11}(t)+1)dt\nonumber\\
&\ \ \ \ \ \ \ \ \ \ \ \ \ \ \ \ \ \ \ \ \ \ \ \ \ \ \ \ \ \ \ \ \ \ \ \ \ \ -\lambda_1(t)d\overline{\lambda_2}(t)-\lambda_2(t)d\overline{\lambda_1}(t)-\overline{\lambda_1}(t)d\lambda_2(t)-\overline{\lambda_2}(t)d\lambda_1(t),\nonumber\\
&d|\lambda_1-\lambda_2|^2(t)=(\lambda_1(t)-\lambda_2(t))d(\overline{\lambda_1}(t)-\overline{\lambda_2}(t))+(\overline{\lambda_1}(t)-\overline{\lambda_2}(t))d(\lambda_1(t)-\lambda_2(t))+2(2\Ov_{11}(t)-1)dt\nonumber,
\end{align}
\begin{eqnarray}
\begin{aligned}\label{437}
d\Ov_{11}(t)&=\frac{d\bigl(||J||^2_2-\lambda_1\overline{\lambda_2}-\overline{\lambda_1}\lambda_2\bigr)(t)-\Ov_{11}(t)d|\lambda_1-\lambda_2|^2(t)}{|\lambda_1(t)-\lambda_2(t)|^2}\\\
&+\frac{\Ov_{11}(t)d\langle |\lambda_1-\lambda_2|^2\rangle_t-d\langle ||J||^2_2-\lambda_1\overline{\lambda_2}-\overline{\lambda_1}\lambda_2,|\lambda_1-\lambda_2|^2\rangle_t}{|\lambda_1(t)-\lambda_2(t)|^4}.
\end{aligned}
\end{eqnarray}
We first deal with the local martingale part of \eqref{437}. From the above two stochastic differential equations and \eqref{42}, we have 
\begin{align}\label{438}
&({\rm local\ martingale\ term\ of}\ d\Ov_{11})\nonumber\\
&=\frac{2}{|\lambda_1(t)-\lambda_2(t)|^2}
\re\biggl(\sum_{i,j=1}^2J_{ij}(t)d\overline{J_{ij}}(t)-\Ov_{11}(t)\frac{\overline{\lambda_1}(t)-\overline{\lambda_2}(t)}{\lambda_1(t)-\lambda_2(t)}\bigl(2J_{21}(t)dJ_{12}(t)+2J_{12}(t)dJ_{21}(t)\bigr)\nonumber\\
&-\Ov_{11}(t)\frac{\overline{\lambda_1}(t)-\overline{\lambda_2}(t)}{\lambda_1(t)-\lambda_2(t)}
\bigl((\lambda_1(t)+\lambda_2(t)-2J_{22}(t))dJ_{11}(t)+(\lambda_1(t)+\lambda_2(t)-2J_{11}(t))dJ_{22}(t)\bigr)\nonumber\\
&+\frac{\overline{\lambda_1}(t)}{\lambda_1(t)-\lambda_2(t)}\bigl((\lambda_2(t)-J_{22}(t))dJ_{11}(t)+J_{21}(t)dJ_{12}(t)+J_{12}(t)dJ_{21}(t)+(\lambda_2(t)-J_{11}(t))dJ_{22}(t)\bigr)\nonumber\\
&-\frac{\overline{\lambda_2}(t)}{\lambda_1(t)-\lambda_2(t)}\bigl((\lambda_1(t)-J_{22}(t))dJ_{11}(t)+J_{21}(t)dJ_{12}(t)+J_{12}(t)dJ_{21}(t)+(\lambda_1(t)-J_{11}(t))dJ_{22}(t)\bigr)
\biggr)\nonumber\\
&=\frac{2}{|\lambda_1(t)-\lambda_2(t)|^2}\re\biggl(\sum_{i,j=1}^2J_{ij}(t)d\overline{J_{ij}}(t)\nonumber\\
&+(2\Ov_{11}(t)-1)\frac{\overline{\lambda_1}(t)-\overline{\lambda_2}(t)}{\lambda_1(t)-\lambda_2(t)}\bigl(J_{22}(t)dJ_{11}(t)+J_{11}(t)dJ_{22}(t)-J_{21}(t)dJ_{12}(t)-J_{12}(t)dJ_{21}(t)\bigr)\nonumber\\
&+\frac{\overline{\lambda_1}(t)\lambda_2(t)-\lambda_1(t)\overline{\lambda_2}(t)-\Ov_{11}(t)(\lambda_1(t)+\lambda_2(t))(\overline{\lambda_1}(t)-\overline{\lambda_2}(t))}{\lambda_1(t)-\lambda_2(t)}(dJ_{11}(t)+dJ_{22}(t))
\biggr).
\end{align}
The above equation gives the explicit form of $M_{11}(t)$ in the Corollary \ref{cor2}. Next, we calculate the drift term of \eqref{437}. From \eqref{436}, \eqref{435}, \eqref{corr} and the elementary relations 
$$\tr(J(t))=\lambda_1(t)+\lambda_2(t),\det(J(t))=\lambda_1(t)\lambda_2(t),$$
 we obtain the real quadratic variations: 
\begin{align}\label{439}
d\langle |\lambda_1-\lambda_2|^2\rangle_t
&=\Bigl(4(2\Ov_{11}(t)-1)|\lambda_1(t)-\lambda_2(t)|^2+2\tau\bigl((\lambda_1(t)-\lambda_2(t))^2+(\overline{\lambda_1}(t)-\overline{\lambda_2}(t))^2\bigr)\Bigr)dt,\\
d\langle ||J||^2_2,\lambda_1\rangle_t
&=\frac{1}{\lambda_1(t)-\lambda_2(t)}\nonumber\\
&\times\biggl(
(\lambda_1(t)-J_{22}(t))J_{11}(t)+J_{21}(t)J_{12}(t)
+J_{12}(t)J_{21}(t)+(\lambda_1(t)-J_{11}(t))J_{22}(t)\nonumber\\
&+\tau(\lambda_1(t)-J_{22}(t))\overline{J_{11}}(t)+\tau |J_{21}(t)|^2
+\tau |J_{12}(t)|^2+\tau(\lambda_1(t)-J_{11}(t))\overline{J_{22}}(t)
\biggr)dt\nonumber\\
&=\lambda_1(t)dt+\tau\frac{||J(t)||^2_2-\lambda_2(t)(\overline{\lambda_1}(t)+\overline{\lambda_2}(t))}{\lambda_1(t)-\lambda_2(t)}dt,\nonumber\\
d\langle ||J||^2_2,\lambda_2\rangle_t&=\lambda_2(t)dt+\tau\frac{||J(t)||^2_2-\lambda_1(t)(\overline{\lambda_1}(t)+\overline{\lambda_2}(t))}{\lambda_2(t)-\lambda_1(t)}dt.\nonumber
\end{align}
By using the equation 
$$2||J(t)||^2_2-|\lambda_1(t)+\lambda_2(t)|^2+|\lambda_1(t)-\lambda_2(t)|^2=2\Ov_{11}(t)|\lambda_1(t)-\lambda_2(t)|^2,$$
 we get
\begin{align}\label{440}
&d\langle ||J||^2_2-\lambda_1\overline{\lambda_2}-\overline{\lambda_1}\lambda_2,|\lambda_1-\lambda_2|^2\rangle_t\nonumber\\
&=d\langle ||J||^2_2,|\lambda_1-\lambda_2|^2\rangle_t-d\langle \lambda_1\overline{\lambda_2},|\lambda_1-\lambda_2|^2\rangle_t-d\langle \overline{\lambda_1}\lambda_2,|\lambda_1-\lambda_2|^2\rangle_t\nonumber\\
&=2\re\Biggl\{(\overline{\lambda_1}(t)-\overline{\lambda_2}(t))\left(\lambda_1(t)-\lambda_2(t)+\tau\frac{2||J(t)||^2_2-|\lambda_1(t)+\lambda_2(t)|^2)}{\lambda_1(t)-\lambda_2(t)}\right)
+\tau(\lambda_1(t)-\lambda_2(t))\lambda_1(t)\nonumber\\
&\ \ \ \ \ \ \ \ \ \ \ +(2\Ov_{11}(t)-1)(|\lambda_1(t)|^2-2\lambda_1(t)\overline{\lambda_2}(t)+|\lambda_2(t)|^2)
-\tau(\overline{\lambda_1}(t)-\overline{\lambda_2}(t))\overline{\lambda_2}(t)\Biggr\}dt\nonumber\\
&=4\Ov_{11}(t)|\lambda_1(t)-\lambda_2(t)|^2dt+2\tau\Ov_{11}(t)((\lambda_1(t)-\lambda_2(t))^2+(\overline{\lambda_1}(t)-\overline{\lambda_2}(t))^2)dt.
\end{align}
From \eqref{42}, \eqref{437}, \eqref{439} and \eqref{440}, we obtain 
\begin{align*}
&({\rm drift\ term\ of}\ d\Ov_{11})\\
&=\frac{1}{|\lambda_1(t)-\lambda_2(t)|^2}\left(2(\Ov_{11}(t)+1)+2\tau\re\biggl(\frac{\overline{\lambda_1}(t)-\overline{\lambda_2}(t)}{\lambda_1(t)-\lambda_2(t)}\biggr)\right)dt\\
&-\frac{\Ov_{11}(t)}{|\lambda_1(t)-\lambda_2(t)|^2}\left(2(2\Ov_{11}(t)-1)+4\tau\re\biggl(\frac{\overline{\lambda_1}(t)-\overline{\lambda_2}(t)}{\lambda_1(t)-\lambda_2(t)}\biggr)\right)dt\\
&+\frac{\Ov_{11}(t)}{{|\lambda_1(t)-\lambda_2(t)|^4}}\Bigl(4(2\Ov_{11}(t)-1)|\lambda_1(t)-\lambda_2(t)|^2+2\tau((\lambda_1(t)-\lambda_2(t))^2+(\overline{\lambda_1}(t)-\overline{\lambda_2}(t))^2)\Bigr)dt\\
&-\frac{1}{{|\lambda_1(t)-\lambda_2(t)|^4}}\biggl(4\Ov_{11}(t)|\lambda_1(t)-\lambda_2(t)|^2+2\tau\Ov_{11}(t)((\lambda_1(t)-\lambda_2(t))^2+(\overline{\lambda_1}(t)-\overline{\lambda_2}(t))^2)\biggr)dt\\
&=\frac{1}{|\lambda_1(t)-\lambda_2(t)|^2}\left((2\Ov_{11}(t)-1)^2+1-2\tau(2\Ov_{11}(t)-1)\re\biggl(\frac{\overline{\lambda_1}(t)-\overline{\lambda_2}(t)}{\lambda_1(t)-\lambda_2(t)}\biggr)\right)dt
\end{align*}
which gives the drift term of Corollary \ref{cor2}. Thirdly, we calculate the real quadratic variation $\langle \Ov_{11}\rangle_t$. We define 
\begin{align*}
&A:=(2\Ov_{11}(t)-1)\frac{\overline{\lambda_1}(t)-\overline{\lambda_2}(t)}{\lambda_1(t)-\lambda_2(t)},\\
&B:=\frac{\overline{\lambda_1}(t)\lambda_2(t)-\lambda_1(t)\overline{\lambda_2}(t)-\Ov_{11}(t)(\lambda_1(t)+\lambda_2(t))(\overline{\lambda_1}(t)-\overline{\lambda_2}(t))}{\lambda_1(t)-\lambda_2(t)}
\end{align*}
and rewrite  \eqref{438} as 
\begin{align*}
&({\rm local\ martingale\ term\ of}\ d\Ov_{11})\\
&=\frac{1}{|\lambda_1(t)-\lambda_2(t)|^2}\\
&\times\biggl(
(\overline{J_{11}}(t)+AJ_{22}(t)+B)dJ_{11}(t)+(\overline{J_{22}}(t)+AJ_{11}(t)+B)dJ_{22}(t)
+(\overline{J_{12}}(t)-AJ_{21}(t))dJ_{12}(t)\\
&\ \ \ +(\overline{J_{21}}(t)-AJ_{12}(t))dJ_{21}(t)
+(J_{11}(t)+\overline{A}\overline{J_{22}}(t)+\overline{B})d\overline{J_{11}}(t)+(J_{22}(t)+\overline{A}\overline{J_{11}}(t)+\overline{B})d\overline{J_{22}}(t)\\
&\ \ \ +(J_{12}(t)-\overline{A}\overline{J_{21}}(t))d\overline{J_{12}}(t)+(J_{21}(t)-\overline{A}\overline{J_{12}}(t))d\overline{J_{21}}(t)
\biggr).
\end{align*}
By using \eqref{corr} and the above notations,
\begin{align*}
&d\langle \Ov_{11}\rangle_t
=\frac{1}{|\lambda_1(t)-\lambda_2(t)|^4}\\
&\times\biggl(
\tau(\overline{J_{11}}(t)+AJ_{22}(t)+B)^2
+\tau(\overline{J_{22}}(t)+AJ_{11}(t)+B)^2
+\tau(J_{11}(t)+\overline{A}\overline{J_{22}}(t)+\overline{B})^2\\
&\ \ \ \ \ \ +\tau(J_{22}(t)+\overline{A}\overline{J_{11}}(t)+\overline{B})^2
+2|\overline{J_{11}}(t)+AJ_{22}(t)+B|^2
+2|\overline{J_{22}}(t)+AJ_{11}(t)+B|^2\\
&\ \ \ \ \ \ +2|\overline{J_{12}}(t)-AJ_{21}(t)|^2
+2|\overline{J_{21}}(t)-AJ_{12}(t)|^2
+2\tau(\overline{J_{12}}(t)-AJ_{21}(t))(\overline{J_{21}}(t)-AJ_{12}(t))\\
&\ \ \ \ \ \ +2\tau(J_{12}(t)-\overline{A}\overline{J_{21}}(t))(J_{21}(t)-\overline{A}\overline{J_{12}}(t))
\biggr)dt\\
&=\frac{2}{|\lambda_1(t)-\lambda_2(t)|^4}
\biggl(||J(t)||^2_2(|A|^2+1)+2|B|^2+2A\lambda_1(t)\lambda_2(t)+2\overline{A}\overline{\lambda_1}(t)\overline{\lambda_2}(t)\\
&+B\Bigl(\lambda_1(t)+\lambda_2(t)+\overline{A}(\overline{\lambda_1}(t)+\overline{\lambda_2}(t))\Bigr)
+\overline{B}\Bigl(\overline{\lambda_1}(t)+\overline{\lambda_2}(t)+A(\lambda_1(t)+\lambda_2(t))\Bigr)
\biggr)dt\\
&+\frac{2\tau}{|\lambda_1(t)-\lambda_2(t)|^4}\\
&\times\re\biggl((A^2+1)(J_{11}(t)^2+J_{22}(t)^2+2J_{12}(t)J_{21}(t))+2B\Bigl(\overline{\lambda_1}(t)+\overline{\lambda_2}(t)+A(\lambda_1(t)+\lambda_2(t))\Bigr)\\
&+2A(J_{11}(t)\overline{J_{22}}(t)+\overline{J_{11}}(t)J_{22}(t)-|J_{12}(t)|^2-|J_{21}(t)|^2)+2B^2
\biggr)dt.
\end{align*}
Because 
\begin{align*}
&-2B=\overline{\lambda_1}(t)+\overline{\lambda_2}(t)+A(\lambda_1(t)+\lambda_2(t)),\ \  J_{11}(t)^2+J_{22}(t)^2+2J_{12}(t)J_{21}(t)=\lambda_1(t)^2+\lambda_2(t)^2,\\
&J_{11}(t)\overline{J_{22}}(t)+\overline{J_{11}}(t)J_{22}(t)-|J_{12}(t)|^2-|J_{21}(t)|^2=|\lambda_1(t)+\lambda_2(t)|^2-||J(t)||^2_2,
\end{align*}
we obtain
\begin{align*}
\{ {\rm first\ term\ of}\ d\langle \Ov_{11}\rangle_t\}
&=\frac{4\Ov_{11}(t)(2\Ov_{11}(t)-1)(\Ov_{11}(t)-1)}{|\lambda_1(t)-\lambda_2(t)|^2}dt,\\
\{ {\rm second\ term\ of}\ d\langle \Ov_{11}\rangle_t\}
&=-\frac{4\tau \Ov_{11}(t)(\Ov_{11}(t)-1)\re\bigl((\lambda_1(t)-\lambda_2(t))^2\bigr)}{|\lambda_1(t)-\lambda_2(t)|^4}dt.
\end{align*}
Combining both, we obtain $d\langle \Ov_{11}\rangle_t$. Finally, from \eqref{439} and \eqref{440}, we yield
\begin{align*}
&d\langle \Ov_{11},|\lambda_1-\lambda_2|^2\rangle_t=\frac{d\langle ||J||^2_2-\lambda_1\overline{\lambda_2}-\overline{\lambda_1}\lambda_2,|\lambda_1-\lambda_2|^2\rangle_t-\Ov_{11}(t)d\langle |\lambda_1-\lambda_2|^2\rangle_t}{|\lambda_1(t)-\lambda_2(t)|^2}\\
&=\frac{1}{|\lambda_1(t)-\lambda_2(t)|^2}\biggl(4\Ov_{11}(t)|\lambda_1(t)-\lambda_2(t)|^2+2\tau\Ov_{11}(t)\bigl((\lambda_1(t)-\lambda_2(t))^2+(\overline{\lambda_1}(t)-\overline{\lambda_2}(t))^2\bigr)\\
&\ \ \ \ \ \ -\Ov_{11}(t)\Bigl(4(2\Ov_{11}(t)-1)|\lambda_1(t)-\lambda_2(t)|^2+2\tau\bigl((\lambda_1(t)-\lambda_2(t))^2+(\overline{\lambda_1}(t)-\overline{\lambda_2}(t))^2\bigr)\Bigr)
\biggr)dt\\
&=-8\Ov_{11}(t)(\Ov_{11}(t)-1)dt.
\end{align*}
Therefore, we finish the proof of Corollary \ref{cor2}.


\appendix
\section{Appendix: Tools and basic properties}
In this appendix, we record some elementary properties regarding eigenvalues and determinants. We also review complex Ito's formula.
\begin{lemma}[Drift terms of Dyson's model and characteristic polynomial]\label{lemA1}\ \\
Assume that $N\times N$ matrix $A$ has simple spectrum $\lambda_1,\cdots,\lambda_N$, and the characteristic polynomial is $f(\lambda):=\det(\lambda I_N-A)$. Then for any $i=1,\cdots,N$,
\begin{align}\label{A1}
\frac{f_{\lambda \lambda}(\lambda_i)}{f_\lambda(\lambda_i)}=2\sum_{j(\neq i)}\frac{1}{\lambda_i-\lambda_j}.
\end{align}
\end{lemma}
\begin{proof}
By definition of eigenvalues, 
$f(\lambda)=\prod_{j=1}^N(\lambda-\lambda_j)=(\lambda-\lambda_i)\prod_{j (\neq i)}(\lambda-\lambda_j)$. Differentiating this with respect to $\lambda$, we have
\begin{align*}
f_\lambda(\lambda)&=\prod_{j (\neq i)}(\lambda-\lambda_j)+(\lambda-\lambda_i)\sum_{j (\neq i)}\prod_{k (\neq i,j)}(\lambda-\lambda_k),\\
f_{\lambda \lambda}(\lambda)&=\sum_{j (\neq i)}\prod_{k (\neq i,j)}(\lambda-\lambda_k)+\sum_{j (\neq i)}\prod_{k (\neq i,j)}(\lambda-\lambda_k)+(\lambda-\lambda_i)\sum_{j (\neq i)}\sum_{k (\neq i,j)}\prod_{\ell (\neq i,j,k)}(\lambda-\lambda_\ell).
\end{align*}
Substituting $\lambda=\lambda_i$,
\begin{align*}
f_\lambda(\lambda_i)=\prod_{j (\neq i)}(\lambda_i-\lambda_j),\ 
f_{\lambda \lambda}(\lambda_i)=2\sum_{j (\neq i)}\prod_{k (\neq i,j)}(\lambda_i-\lambda_k).
\end{align*}
On the other hand, the right hand side of (A.1) is 
\begin{align*}
\sum_{j(\neq i)}\frac{1}{\lambda_i-\lambda_j}=\frac{\sum_{j (\neq i)}\prod_{k (\neq i,j)}(\lambda_i-\lambda_k)}{\prod_{j (\neq i)}(\lambda_i-\lambda_j)},
\end{align*}
and the claim holds.
\end{proof}

\begin{lemma}[Minor determinants and eigenvalues]\label{lemA2}\ \\
Assume that $N \times N$ matrix $A$ has the eigenvalues $\lambda_1,\cdots,\lambda_N$. Then for all $k=1,\cdots,N$, 
\begin{align}\label{A2}
\sum_{1\le j_1<\cdots<j_k\le N}\ \prod_{\ell=1}^k\lambda_{j_\ell}=\sum_{1\le j_1<\cdots<j_k\le N}\underset{1\le \ell,m\le k}{{\rm det}}(A_{j_\ell j_m})
\end{align}
where $\underset{1\le \ell,m\le k}{{\rm det}}(A_{j_\ell j_m})$ is the $k$-th principal minor indexed by $\{ j_1<\cdots<j_k\} \subset \{1,\cdots,N\}$: 
\begin{align*}
\underset{1\le \ell,m\le k}{{\rm det}}(A_{j_\ell j_m})=\det
\begin{pmatrix}
a_{j_1j_1} & a_{j_1j_2} & \cdots & a_{j_1j_k}\\
a_{j_2j_1} & a_{j_2j_2} & \cdots & a_{j_2j_k}\\
\vdots & & \ddots & \vdots\\
a_{j_kj_1} & a_{j_kj_2} & \cdots & a_{j_kj_k}
\end{pmatrix}.
\end{align*}
\end{lemma}
\begin{proof}
Applying  binomial expansion to the characteristic polynomial $f(\lambda)=\det(\lambda I_N-A)$, we have
\begin{align*}
f(\lambda)=\lambda^N+\sum_{k=1}^N(-1)^k\lambda^{N-k}\sum_{1\le j_1<\cdots<j_k\le N}\ \prod_{\ell=1}^k\lambda_{j_\ell}.
\end{align*}
We also apply Fredholm determinant expansion to $f(\lambda)$ and obtain 
\begin{align*}
f(\lambda)=\lambda^N+\sum_{k=1}^{N}(-1)^k\lambda^{N-k}\sum_{1\le j_1<\cdots<j_k\le N}\underset{1\le \ell,m\le k}{{\rm det}}(A_{j_\ell j_m}).
\end{align*}
Therefore, the claim holds by comparing the coefficient of $\lambda^{N-k}$ each other. 
\end{proof}

\begin{lemma}[Twice cofactor expansion form]\label{lemA3}\ \\ 
For $N \times N$ matrix $A$ and the fixed integers $k<\ell$, 
\begin{align}\label{A3}
\det A&=a_{kk}\det(A_{k|k})-a_{k\ell}a_{\ell k}\det(A_{k\ell |\ell k})+\sum_{\substack{q\neq k,\ell\\q<\ell}}(-1)^{k+q-1}a_{k\ell}a_{\ell q}\det(A_{k\ell | \ell q})\nonumber\\
&+\sum_{\substack{q\neq k,\ell\\q>\ell}}(-1)^{k+q}a_{k\ell}a_{\ell q}\det(A_{k\ell | \ell q})+\sum_{\substack{p\neq k,\ell\\p>k}}(-1)^{\ell+p-1}a_{kp}a_{\ell k}\det(A_{k\ell |pk})\nonumber\\
&+\sum_{\substack{p\neq k,\ell\\p<k}}(-1)^{\ell+p}a_{kp}a_{\ell k}\det(A_{k\ell |pk})+\sum_{\substack{p\neq k,\ell,\ q \neq k \\p>q}}(-1)^{k+\ell+p+q-1}a_{kp}a_{\ell q}\det(A_{k\ell |pq})\nonumber\\
&+\sum_{\substack{p\neq k,\ell,\ q \neq k \\p<q}}(-1)^{k+\ell+p+q}a_{kp}a_{\ell q}\det(A_{k\ell |pq})
\end{align}
where $A_{k\ell |pq}$ is the $(N-2) \times (N-2)$ minor matrix that is obtained by removing the $k,\ell$-th rows and the $p,q$-th columns from A.
\end{lemma}
\begin{proof}
We expand $\det A$ by the $k$-th row, and we also expand each of the $(N-1)$-th determinants by the $\ell$-th column.
\end{proof}

\begin{lemma}[Cauchy-Binet formula, \cite{HJ}]\label{lemA4}\ \\
Let $A \in M_{m,n}(\C)$. For index sets $\alpha=\{i_1,\cdots,i_p\} \subseteq \{1,\cdots,m\},\ p \le m$\ and $\beta=\{j_1,\cdots, j_q\} \subseteq \{1,\cdots, n\},\ q \le n$,\ the $p \times q$ submatrix $A(\alpha,\beta)$ is defined as $A(\alpha,\beta)_{r,s}:=A_{i_r,j_s}$. Here, $i_1<\cdots<i_p$ and $j_1<\cdots<j_q$ hold and these index sets are ordered lexicographically. When $\sharp\alpha = \sharp\beta = k \le {\rm min}\{m,n\}$, the $k$-th compound matrix of $A$ is defined as the $\binom{m}{k} \times \binom{n}{k}$ matrix whose $(\alpha, \beta)$ entry is $\det(A(\alpha,\beta))$, and we denote this by $C_k(A)$.\\
Let $A \in M_{m,k}(\C)$,\ $B \in M_{k,n}(\C)$ and $C:=AB$. We fix the index sets $\alpha \subseteq \{1,\cdots,m\}$ and $\beta \subseteq \{1,\cdots, n\}$, where $\sharp\alpha = \sharp\beta =r \le {\rm min}\{m,k,n\}$. Then the determinant of the submatrix $C(\alpha, \beta)$ has an expression:
\begin{align}
\det(C(\alpha, \beta))=\sum_{\gamma}\det(A(\alpha,\gamma))\det(B(\gamma,\beta))
\end{align}
where the summation is taken over all index sets $\gamma \subseteq \{1,\cdots,k\}$ of cardinality $r$.
\end{lemma}

\begin{lemma}[Ito's formula for complex cases]\label{lemA5}\ \\
Suppose that $Z_t=(Z^{(1)}_t,\cdots,Z^{(n)}_t)$ is a continuous complex semi-martingale vector and $f: \C^n \to C$ is a $C^2$ function. Then
\begin{align}\label{A7}
df(Z_t)&=\sum_{i=1}^n\Bigl(\partial_{z_i}f(Z_t)dZ^{(i)}_t+\overline{\partial_{z_i}}f(Z_t)d\overline{Z}^{(i)}_t\Bigr) \nonumber\\
&+\frac{1}{2}\sum_{i=1}^n\Bigl(\partial_{z_i}\partial_{z_i}f(Z_t)d\langle Z^{(i)},Z^{(i)}\rangle_t+2\partial_{z_i}\overline{\partial_{z_i}}f(Z_t)d\langle Z^{(i)},\overline{Z}^{(i)}\rangle_t+\overline{\partial_{z_i}}\ \overline{\partial_{z_i}}f(Z_t)d\langle \overline{Z}^{(i)},\overline{Z}^{(i)}\rangle_t\Bigr)\nonumber\\
&+\sum_{i<j}\Bigl(\partial_{z_i}\partial_{z_j}f(Z_t)d\langle Z^{(i)},Z^{(j)}\rangle_t+\overline{\partial_{z_i}}\partial_{z_j}f(Z_t)d\langle \overline{Z}^{(i)},Z^{(j)}\rangle_t+\partial_{z_i}\overline{\partial_{z_j}}f(Z_t)d\langle Z^{(i)},\overline{Z}^{(j)}\rangle_t\nonumber\\
&\ \ \ \ \ \ \ \ \ \ +\overline{\partial_{z_i}}\ \overline{\partial_{z_j}}f(Z_t)d\langle \overline{Z}^{(i)},\overline{Z}^{(j)}\rangle_t\Bigr)
\end{align}
where the complex quadratic variation $d\langle \cdot,\cdot\rangle_t$ is defined by \eqref{cqv}.
\end{lemma}
\begin{proof}
We apply standard Ito's formula for $f : \R^{2n} \to \R^2$ and put the real and imaginary part together algebraically.
\end{proof}

                 
\section*{Acknowledgments}
The author would like to thank Prof. Hideki Tanemura for his valuable discussions and encouragement. He also would like to thank Prof. Makoto Katori and Dr. Mikio Shibuya for their useful comments.



\begin{thebibliography}{99}

\bibitem{Oxford}G. Akemann, J. Baik and P. D. Francesco, \textit{The Oxford Handbook of Random Matrix Theory}, Oxford University Press (2011)


\bibitem{AP} G. Akemann and M. J. Phillips, \textit{Universality Conjecture for all Airy, Sine and Bessel Kernels in the Complex Plane}, Random Matrices, MSRI Publications 65, 1-24 (2014)


\bibitem{AGZ} G. W. Anderson, A. Guionnet, O. Zeitouni, \textit{An Introduction to Random Matrices}, Cambridge University Press (2005)


\bibitem{OU} J. P. Blaizot, J. Grela, M. A. Nowak, W. Tarnowski and P. Warcho\l, \textit{Ornstein-Uhlenbeck diffusion of hermitian and non-hermitian matrices-unexpected links}, J. Stat. Mech. 054037(2016) 


\bibitem{BD} P. Bourgade and G. Dubach, \textit{The distribution of overlaps between eigenvectors of Ginibre matrices}, Probab. Theory Related Fields, DOI:10.1007/s00440-019-00953-x, 14 Nov. 2019


\bibitem{Bru}M. F. Bru, \textit{Diffusions of perturbed principal component analysis}, J. Multivariate Anal. 29, 1, 127-136 (1989)


\bibitem{Bru2}M. F. Bru, \textit{Wishart Processes}, J. Theoret. Probab. 4, 725-751 (1991)


\bibitem{Dysonian} Z. Burda, J. Grela, M. A. Nowak, W. Tarnowski, and P. Warcho\l, \textit{Dysonian Dynamics of the Ginibre Ensemble}, Phys. Rev. Lett. 113, 104102 (2014)


\bibitem{CM} J. T. Chalker and B. Mehlig, \textit{Eigenvector statistics in non-Hermitian random matrix ensembles}, Phys. Rev. Lett. 81, 3367-3370 (1998)

\bibitem{CS} A. Crisanti and H. Sompolinsky, \textit{Dynamics of spin systems with randomly asymmetric bonds: Langevin dynamics and a spherical model}, Phys. Rev. A 36, 4922-4939 (1987).

\bibitem{Dyson} F. J. Dyson, \textit{A Brownian-motion model for the eigenvalues of a random matrix}, J. Math. Phys. 3, 1192-1198 (1962)


\bibitem{Fyo} Y. V. Fyodorov, \textit{On Statistics of Bi-Orthogonal Eigenvectors in Real and Complex Ginibre Ensembles: Combining Partial Schur Decomposition with Supersymmetry}, Comm. in Math. Phys. 363, 579-603 (2018)


\bibitem{FKS} Y. V. Fyodorov, B. A. Khoruzhenko and H. J. Sommers, \textit{Almost-Hermitian Random Matrices: Eigenvalue Density in the Complex Plane}, Phys. Lett. A 226, 46-52 (1997)


\bibitem{FW} Y. V. Fyodorov and W. Tarnowski, \textit{Condition numbers for real eigenvalues in real Elliptic Gaussian ensemble}, arXiv:1910.09204v1 (2019)


\bibitem{Gin} J. Ginibre, \textit{Statistical ensembles of complex, quaternion, and real matrices}, J. Math. Phys. 6, 3, 440-449 (1965)


\bibitem{Girko} V. L. Girko, \textit{The elliptic law}, Theory Probab. Appl. 30, 4, 640-651 (1985)


\bibitem{GM} P. Graczyk and J. Ma\l ecki, \textit{On squared Bessel particle systems}, Bernoulli 25, 2, 828-847 (2019)


\bibitem{GW} J. Grela and P. Warcho\l, \textit{Full Dysonian dynamics of the complex Ginibre ensemble}, J. Phys. A: Math. Theor. 51, 42 (2018)


\bibitem{HJ} R. A. Horn and C. R. Johnson, \textit{Matrix Analysis}, Cambridge University Press (1999)


\bibitem{KS} I. Karatzas and S. Shreve, \textit{Brownian motion and stochastic calculus}, Springer (1991)


\bibitem{KT1} M. Katori and H. Tanemura, \textit{Scaling limit of vicious walks and two-matrix model}, Phys. Rev. E 66, 011105 (2002)


\bibitem{KT2} M. Katori and H. Tanemura, \textit{Noncolliding Brownian motions and Harish-Chandra formula}, Elect. Comm. in Probab. 8, 112-121 (2003) 


\bibitem{KT3} M. Katori and H. Tanemura, \textit{Symmetry of matrix-valued stochastic processes and noncolliding diffusion particle systems}, J. Math. Phys. 45, 8, 3058-3085 (2004) 


\bibitem{LS} N. Lehmann and H. J. Sommers,  \textit{Eigenvalue Statistics of Random Real Matrices}, Phys. Rev. Lett. 67, 8, 941-944 (1991)


\bibitem{McKean} H. P. McKean, \textit{Stochastic Integrals}, Academic Press, New York (1969)


\bibitem{Mehta} M. L. Mehta, \textit{Random Matrices, third edition}, Academic Press (2004)


\bibitem{RS} L.C.G. Rogers and Z. Shi, \textit{Interacting Brownian particles and the Wigner law}, Probab. Theory  Related  Fields 95, 555-570 (1993)

\bibitem{SCSS} H. J. Sommers, A. Crisanti, H. Sompolinsky and Y. Stein, \textit{Spectrum of large random asymmetric matrices}, Phys. Rev. Lett. 60, 1895-1898 (1988)

\bibitem{Tao} T. Tao, \textit{Topics in Random Matrix Theory}, American Mathematical Society (2012)


\bibitem{TE}L. N. Trefethen and M. Embree, \textit{Spectra and Pseudospectra,The Behavior of Nonnormal Matrices and Operators}, Princeton University Press (2005)


\end{thebibliography}
\end{document}